\documentclass[11pt, reqno, twoside, makeidx]{amsart}
\usepackage[margin=2.7cm, marginpar=1cm]{geometry}

\usepackage{wrapfig}
\usepackage{marginnote}
\usepackage{hyperref}
\usepackage[english]{babel}
\usepackage[utf8]{inputenc}
\usepackage[T1]{fontenc}
\usepackage{amsmath,amsthm,amssymb,amsfonts}
\numberwithin{equation}{section}
\usepackage{enumerate}
\usepackage{faktor}
\usepackage{dsfont}
\usepackage{scalerel,stackengine}
\usepackage{textcomp}
\usepackage{graphicx}
\usepackage{extarrows}
\usepackage{epigraph}
\usepackage{graphicx}
\usepackage{subcaption}
\usepackage{sidecap}
\usepackage{indentfirst}
\usepackage{tikz}
\usepackage{tikz-cd}
\usepackage{tkz-euclide}
\usetikzlibrary{shapes.misc,arrows,matrix,patterns,decorations.markings,positioning}
\tikzset{cross/.style={cross out, draw=black, minimum size=2*(#1-\pgflinewidth), inner sep=0pt, outer sep=0pt},
cross/.default={4.5pt}}
\usepackage{pgfplots}
\usepackage{caption}
\usepackage{float}
\usepackage{color}
\usepackage{epstopdf}
\usepackage{epsfig}
\usepackage{stmaryrd}
\usepackage{color}
\usepackage{geometry}
\usepackage{multicol}
\usepackage{verbatim}
\usepackage{wrapfig}
\usepackage{float}

\DeclareMathOperator{\Ker}{Ker }
\DeclareMathOperator{\Imm}{Im }
\DeclareMathOperator{\tb}{tb}
\DeclareMathOperator{\rot}{rot}
\DeclareMathOperator{\self}{sl}

\DeclareMathOperator{\Spin}{Spin}

\DeclareMathOperator{\height}{height_{\mathcal F}}
\DeclareMathOperator{\xist}{\xi_{std}}
\renewcommand{\geq}{\geqslant}
\renewcommand{\leq}{\leqslant} 
\renewcommand{\epsilon}{\varepsilon}

\newcommand{\N}{\mathbb{N}}
\newcommand{\Z}{\mathbb{Z}}
\newcommand{\Q}{\mathbb{Q}}
\newcommand{\F}{\mathbb{F}}

\newcommand{\C}{\mathbb C}
\newcommand{\J}{\mathcal J}

\newcommand{\s}{\mathfrak{s}}

\DeclareFontFamily{U}{mathx}{\hyphenchar\font45}
\DeclareFontShape{U}{mathx}{m}{n}{
      <5> <6> <7> <8> <9> <10>
      <10.95> <12> <14.4> <17.28> <20.74> <24.88>
      mathx10
      }{}
\DeclareSymbolFont{mathx}{U}{mathx}{m}{n}
\DeclareFontSubstitution{U}{mathx}{m}{n}
\DeclareMathAccent{\widecheck}{0}{mathx}{"71}
\DeclareMathAccent{\wideparen}{0}{mathx}{"75}

\newtheorem{teo}{Theorem}[section]
\newtheorem*{teo*}{Theorem}
\newtheorem{lemma}[teo]{Lemma}
\newtheorem{prop}[teo]{Proposition}
\newtheorem*{prop*}{Proposition}
\newtheorem{defin}[teo]{Definition}
\newtheorem{cor}[teo]{Corollary}

\newtheorem{remark}[teo]{Remark}

\usepackage{xpatch}
\makeatletter
\xpatchcmd{\@thm}{\thm@headpunct{.}}{\thm@headpunct{}}{}{}
\makeatother

\pgfplotsset{compat=1.18}
\begin{document}
\title[Fillable structures on negative-definite Seifert fibred spaces]{Fillable structures on negative-definite \\ Seifert fibred spaces}
\author{Alberto Cavallo and Irena Matkovi\v c}
\address{HUN-REN Alfr\'ed R\'enyi Institute of Mathematics, Budapest 1053, Hungary}
\email{acavallo@impan.pl}
\address{Uppsala universitet, Uppsala 751 06, Sweden}
\email{irma6504@student.uu.se}


\begin{abstract}
 We classify fillable contact structures on all negative-definite star-shaped plumbings. We show that such Seifert fibred spaces admit a unique negative maximal twisting number and compute it explicitly using the Alexander filtration in lattice cohomology, providing its first Floer-theoretic interpretation. 
 In addition, we show that all the negative-twisting tight structures on these manifolds are induced by the Stein structures on the minimal resolution of the underlying complex surface singularity.
 As an application, we provide a necessary condition for a negative-definite Seifert fibred space to admit a separating contact-type embedding in a strong symplectic filling of a generalised $L$-space. 
\end{abstract}

\maketitle

\thispagestyle{empty} \vspace{-0.85cm}

\section{Introduction}
Investigating contact structures on Seifert fibred spaces provides a bridge between $3$-dimensional topology and algebraic geometry. Indeed, links of normal surface singularities naturally inherit a contact structure, known as the Milnor fillable structure, while their JSJ decomposition consists of Seifert fibred pieces, see \cite{Nemethi0,Nemethi,BP}. In this paper, we adopt a topological perspective and classify all the symplectically fillable contact structures on negative-definite star-shaped plumbings, a special case of which are canonically oriented Brieskorn spheres.

It has long been understood that the behaviour of contact structures is constrained when the underlying manifold is an $L$-space, see \cite{Wu,G-,LS-existance,GLS,Irena}. Despite various indications in the literature, no comparable unifying picture has previously emerged for negative-definite Seifert fibred spaces, beyond the fact that they arise as links of surface singularities and therefore support the Milnor fillable structure. We will see that all fillable structures on these manifolds are essentially of the same type, admitting a Stein filling diffeomorphic to the minimal resolution.

The classification of tight structures usually requires matching between constructions and possible convex decompositions. It goes back to the work of Giroux \cite{Giroux1,Giroux2} and Honda \cite{Honda1,Honda2} on circle bundles that the behaviour of contact structures on fibred manifolds is governed by whether there exists a regular fibre around which the contact structure does not twist. The key insight of the present work is the recognition that a particular convex decomposition due to Ghiggini \cite{G-} and Massot \cite{Massot} translates the existence of contact structures with negative twisting number into a numerical problem. For negative-definite Seifert fibred spaces this problem admits a unique solution, leading to a unified description of all fillable contact structures in the above sense.

Although the Heegaard Floer contact invariant \cite{OSz-contact} has been instrumental in revealing many properties of tight contact structures on Seifert fibred spaces, and the maximal twisting number is their central numerical invariant, it is only here that a Heegaard Floer interpretation of maximal twisting is established. To this end, we introduce a new invariant associated to non-vanishing Heegaard Floer contact classes.

While the behaviour of negative-twisting structures on negative-definite Seifert fibred spaces is shown to be special, the new techniques developed in this paper, which allow us to determine the possible twisting numbers and interpret them as Heegaard Floer invariants, work generally. In fact, in a subsequent paper we use these ideas to obtain a classification of negative twisting structures on arbitrary Seifert fibred spaces, and a correspondence with the Heegaard Floer groups of the underlying 3-manifolds.

\subsection*{Notation and conventions}
The object of our study consists of Seifert fibred spaces with base orbifold homeomorphic to $S^2$. A 3-manifold of this kind, denoted as $M = M(e_0;r_1,\dots,r_n)$ with the integral Euler number $e_0\in \Z$, and $n$ singular fibres with parameters $r_1, \dots , r_n\in(0,1)\cap\Q$, can be represented with a star-shaped plumbing graph $\Gamma$ which we call the \emph{standard graph} of $M$, as shown in Figure \ref{Seifert}. In these diagrams  the Euler number $e_0$ is the framing of the central vertex, while the parameter $r_i$  determines the framings on the $i$-th leg via the negative continued fraction expansion of $-\frac{1}{r_i}$. The construction is explained in full detail in \cite{Saveliev}. Note that  when $n\geq3$ the 3-manifold $M$ uniquely determines the graph $\Gamma$.

\begin{figure}[ht]  \begin{tikzpicture}[scale=0.8]
    \tkzDefPoints{0/0/A, 2/1.5/B, 2/0.5/C, 2/-1.5/D}  \tkzDefPoints{2.5/-0.5/X, 2.5/-0.4/Y, 2.5/-0.6/Z}
    \tkzDrawSegment(A,B)\tkzDrawSegment(A,C)\tkzDrawSegment(A,D)
    \tkzDrawPoints[fill,black,size=5](A,B,C,D)\tkzDrawPoints[fill,black,size=1](X,Y,Z)
     \tkzLabelPoint[above left](A){$e_0$} \tkzLabelPoint[right](B){$-\frac{1}{r_1}$}\tkzLabelPoint[right](C){$-\frac{1}{r_2}$}
     \tkzLabelPoint[right](D){$-\frac{1}{r_n}$}
\end{tikzpicture}
\hspace{2cm}
 \begin{tikzpicture}[scale=0.7]
    \tkzDefPoints{0/0/A, 1.5/1/B, 1.5/-1/D, 3/-1/E, 5/-1/F, 3/1/H, 5/1/I}  \tkzDefPoints{2.25/-0.1/X'', 2.25/0/Y'', 2.25/0.1/Z''}
    \tkzDefPoints{3.5/-1/X, 4.5/-1/Y}\tkzDefPoints{3.9/-1/P, 4/-1/Q, 4.1/-1/R}\tkzDefPoints{3.5/1/X', 4.5/1/Y'}\tkzDefPoints{3.9/1/P', 4/1/Q', 4.1/1/R'}
    \tkzDrawPoints[fill,black,size=1](P,Q,R)\tkzDrawPoints[fill,black,size=1](P',Q',R')\tkzDrawPoints[fill,black,size=1](X'',Y'',Z'')
    \tkzDrawSegment(A,B)\tkzDrawSegment(A,D)\tkzDrawSegment(E,D)\tkzDrawSegment(E,X)\tkzDrawSegment(Y,F)\tkzDrawSegment(B,H)
    \tkzDrawSegment(H,X')\tkzDrawSegment(Y',I)
    \tkzDrawPoints[fill,black,size=5](A,B,D,E,F,H,I)
     \tkzLabelPoint[above left](A){$e_0$} \tkzLabelPoint[above right](B){$m_1^1$}\tkzLabelPoint[above right](H){$m_2^1$}\tkzLabelPoint[above right](I){$m_{k_1}^1$}
     \tkzLabelPoint[below right](D){$m_1^n$}\tkzLabelPoint[below right](E){$m_2^n$}\tkzLabelPoint[below right](F){$m_{k_n}^n$}
\end{tikzpicture}
     \caption{\smaller[1]{The standard graph of $M(e_0;r_1,...,r_n)$ (left). The framings on the $i$-th leg are given by $[m_1^i,\dots,m_{k_i}^i]$, the negative continued fraction expansion of $-\frac{1}{r_i}$ (right). Note that these plumbings are not necessarily negative-definite.}}
     \label{Seifert}
\end{figure}
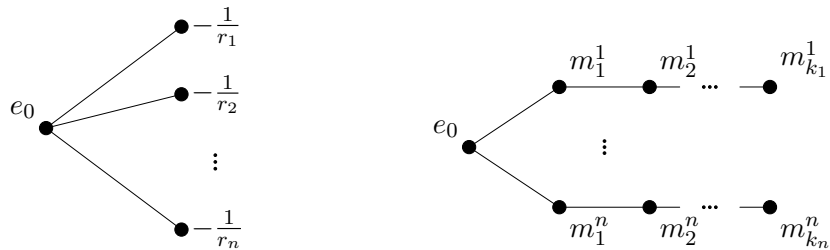   

We also set the convention that every time we consider a 3-manifold presented by a graph, we assume that it inherits the orientation consistent with it. Hence, abusing the notation, we call the Seifert fibred space \emph{negative-definite} if and only if $e(M):=e_0+r_1+\cdots +r_n<0$.

Brieskorn spheres are precisely the Seifert fibred homology spheres. They are the complete intersection singularities, described by exponent parameters $a_i\geq2$, such that $\gcd(a_i,a_j)=1$ for $i\not=j$, as $\Sigma(a_1,...,a_n)$, see \cite[Section 1]{Saveliev} for the exact definition. Note that the Brieskorn parameters $(a_1,...,a_n)$ are related to the Seifert parameters $(e_0;r_1,...,r_n)$ via the equations: 
\begin{equation}
    r_1+...+r_n=-e_0-\dfrac{1}{a_1\cdots a_n},\ \ \ \ \ \  r_i=\dfrac{b_i}{a_i} \ \ \ (b_i\in \Z_{>0})\:, 
    \label{eq:Brieskorn}
\end{equation}
and that  the standard graph $\Gamma$ of $\Sigma(a_1,...,a_n)$ is always negative-definite. For basic definitions, and results regarding plumbing calculus we refer to \cite{GS,Saveliev}.

\subsection*{Classification of negative-twisting tight structures} We recall that the main invariant in the classification of tight contact structures on Seifert fibred spaces is the \emph{maximal twisting number}, which captures the maximal difference between the contact framing and the fibration framing of (any Legendrian realisation of) the regular fibre $K$ of $M$. On negative-definite Seifert fibred spaces we, first, find surgery presentations for all its negative-twisting tight structures, and then show that all the fillable structures are such. The exact number of structures can be computed directly from the Seifert coefficients as shown by Equations \eqref{eq:1} and \eqref{eq:2}. 

\begin{teo}
\label{teo:classification}
 Suppose that $M=M(e_0;r_1,\dots,r_n)$ has negative-definite standard graph. The negative-twisting tight contact structures on $M$ are precisely the Legendrian surgeries on all possible Legendrian realisations of the complete blow-down of the standard graph; moreover, they are all Stein fillable and distinguished, up to isotopy, by their contact invariant $c^+$ in $HF^+(-M)$. 
 
 Furthermore, every structure $\xi$ on $M$ with non-vanishing $c^+(\xi)$ is negative-twisting, and then every symplectically fillable structure on $M$ is as above.
\end{teo}

The Stein domain obtained as the complete blow-down of the standard graph is obtained as follows: starting from the plumbing 4-manifold, at each step there is at most one unknot with framing $-1$, because the graph is negative-definite; hence, we continue to blow down $-1$-unknots until the resulting Kirby presentation contains none of these. The subgraph made by the vertices that disappear during the procedure corresponds to the negative-definite standard graph of a Seifert fibration of a knot complement in $S^3$, that is, the fibration of $T_{p,q}$ for $1\leq p\leq q$ coprime, see \cite{Saveliev}. The resulting rational surgery diagram consists of a torus link $T_{(n-2)p,(n-2)q}$, and possibly two unknots which are cores of the complementary solid tori with respect to the torus supporting the link; we compute the coefficients explicitly in Subsection \ref{subsection:classification}. An example is illustrated in Figure \ref{Picture}. 

\begin{figure}[ht]  \begin{tikzpicture}[scale=0.9]
    \tkzDefPoints{0/0/A, 1.5/1/B, 1.5/0/C, 1.5/-1/D, 3/1/E, 3/0/F} 
    \tkzDrawSegment(A,B)\tkzDrawSegment(A,C)\tkzDrawSegment(A,D) \tkzDrawSegment(B,E)\tkzDrawSegment(C,F)
    \tkzDrawPoints[fill,gray,size=5](A,B,C)\tkzDrawPoints[fill,black,size=5](D)\tkzDrawPoints[fill,blue!65!black,size=5](E)\tkzDrawPoints[fill,red!61!black,size=5](F)
     \tkzLabelPoint[below left](A){$-1$} \tkzLabelPoint[above right](B){$-2$}\tkzLabelPoint[below right](C){$-3$}\tkzLabelPoint[below right](D){$-19$}\tkzLabelPoint[above right](E){$-4$}\tkzLabelPoint[below right](F){$-3$}
 \end{tikzpicture}\hspace{2.5cm}
 \def\svgwidth{0.55\textwidth}
\begingroup%
  \makeatletter%
  \providecommand\color[2][]{%
    \errmessage{(Inkscape) Color is used for the text in Inkscape, but the package 'color.sty' is not loaded}%
    \renewcommand\color[2][]{}%
  }%
  \providecommand\transparent[1]{%
    \errmessage{(Inkscape) Transparency is used (non-zero) for the text in Inkscape, but the package 'transparent.sty' is not loaded}%
    \renewcommand\transparent[1]{}%
  }%
  \providecommand\rotatebox[2]{#2}%
  \newcommand*\fsize{\dimexpr\f@size pt\relax}%
  \newcommand*\lineheight[1]{\fontsize{\fsize}{#1\fsize}\selectfont}%
  \ifx\svgwidth\undefined%
    \setlength{\unitlength}{2000.73146069bp}%
    \ifx\svgscale\undefined%
      \relax%
    \else%
      \setlength{\unitlength}{\unitlength * \real{\svgscale}}%
    \fi%
  \else%
    \setlength{\unitlength}{\svgwidth}%
  \fi%
  \global\let\svgwidth\undefined%
  \global\let\svgscale\undefined%
  \makeatother%
  \begin{picture}(1,0.45156009)%
    \lineheight{1}%
    \setlength\tabcolsep{0pt}%
    \put(0,0){\includegraphics[width=\unitlength,page=1]{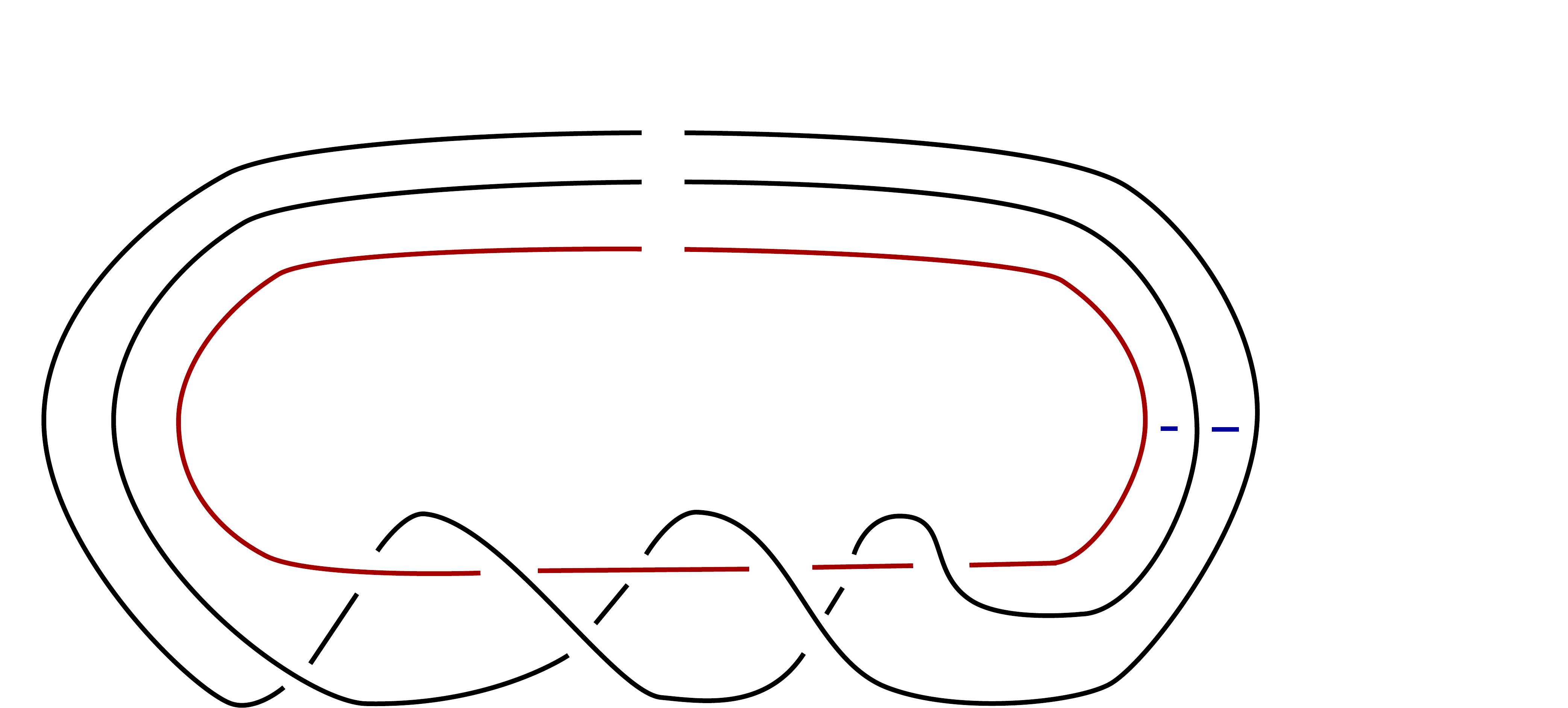}}%
    \put(0.6276221,0.40471469){\color[rgb]{0,0,0}\makebox(0,0)[lt]{\lineheight{1.25}\smash{\begin{tabular}[t]{l}$\textcolor[rgb]{0,0,0.61}{-2}$\end{tabular}}}}%
    \put(0.2305347,0.24717789){\color[rgb]{0,0,0}\makebox(0,0)[lt]{\lineheight{1.25}\smash{\begin{tabular}[t]{l}$\textcolor[rgb]{0.65,0,0}{-2}$\end{tabular}}}}%
    \put(-0.00151481,0.32899465){\color[rgb]{0,0,0}\makebox(0,0)[lt]{\lineheight{1.25}\smash{\begin{tabular}[t]{l}$-13$\end{tabular}}}}%
    \put(0,0){\includegraphics[width=\unitlength,page=2]{Picture.pdf}}%
  \end{picture}%
\endgroup%

     \caption{\smaller[1]{The standard graph of the Brieskorn sphere $\Sigma(7,8,19)$ (left), and its complete blow-down (right). The gray vertices disappear in the procedure; in addition, for the correct Legendrian surgery, we need the trefoil component to be stabilised thirteen times. The manifold then carries 14 fillable contact structures, up to isotopy.}}
     \label{Picture}
 \end{figure}    

Out of classification we read that (in our case) all negative-twisting tight structures are (Stein) fillable. Furthermore, it is possible to show that when $e_0<-1$ and $n=2,3$ our theorem actually covers every possible tight contact structure. Note that when $c^+$ is non-vanishing the contact structure has no half convex Giroux torsion, because otherwise it would be zero-twisting; this means that for negative-definite Seifert fibred spaces half convex torsion obstructs fillability. 

Note that for Seifert fibred spaces with $e_0\leq-n$, which are all $L$-spaces, such a classification result has generally been expected; in fact, with such a restriction the classification was previously established for $n=3,4,$ and it has been in the case of $n=3$ extended also to $e_0=-2$ with negative-definite standard graph \cite{Wu, Shah, Medetogullari,G-,Tosun}. 
On the contrary, for $e_0=-1$ the uniqueness of the twisting number, when $M$ has negative-definite standard graph, takes us by surprise; in fact, in this case the only classification 
previously in the literature is for $M(-1;\frac{1}{2},\frac{1}{3},\frac{k}{6k+1})$ in \cite{MT2}, extended to negative surgeries on positive torus knots in an unpublished note\footnote{To date, Tosun's note only contains a proof that every fillable structure on a negative-definite $M(-1;r_1,r_2,r_3)$ is negative-twisting, and if $M\simeq S^3_{-r}(T_{p,q})$ for $r\in\Q_{>0}$ then it is obtained as described in Theorem \ref{teo:classification}; the techniques strictly depend on the explicit values of the Seifert coefficients of $M$. The note is available at \url{https://bpb-us-e2.wpmucdn.com/sites.ua.edu/dist/8/406/files/2025/09/FillableSurgeriesonTpq.pdf}.}. Note that $L$-spaces with $e_0=-1$, for which the complete classification is known \cite{Irena, GLS}, are never negative-definite; all tight contact structures on these manifolds are actually zero-twisting.

The classification is made possible because we show that there is only one negative twisting number $\text{tw}(M,\xi,K)$ for tight contact structures (on negative-definite Seifert fibred spaces), and we compute its value explicitly. Our understanding is based on the work of Ghiggini \cite{G-}, paired with some subtle arithmetic considerations using Farey sequences.

\subsection*{The height of a full path and the maximal twisting number} 
We introduce a Heegaard Floer interpretation of the maximal twisting number. To establish it, we extract a new invariant from the Ozsv\'ath-Szab\'o full path algorithm which we call the \emph{height}. We show that, for negative-definite Seifert fibred spaces, this invariant completely determines the twisting number of every negative-twisting tight contact structure.

The construction of the height is based on the Alexander grading. Given a characteristic vector $V\in\text{Char}(\Gamma)$, we define the Alexander relative $\mathbb Z$-grading $\mathcal F(V)$ with respect to a distinguished leaf $S_0$ of $\Gamma$, see Subsection \ref{subsection:Alexander}. In the language of lattice cohomology \cite{Nemethi0,Nemethi,Lattice2,Lattice1}, this grading provides a combinatorial counterpart of the Alexander filtration induced in Heegaard Floer homology by the knot represented by $S_0$.

Recall that N\'emethi identifies the cube-grading-zero summand of the lattice cohomology complex with the set of equivalence classes of characteristic vectors arising from the Ozsv\'ath-Szab\'o full path algorithm \cite{OSz-fullpath}; see Section \ref{section:two} for the necessary background. In Lemma \ref{lemA} we show that the grading $\mathcal F$ is generally not constant along a full path. However, it is monotone, and therefore induces a natural filtration on the set of full paths. 

More precisely, we say that one takes an $i$-step from a characteristic vector $V$ if $v_i=-m(i)$, where $m(i)$ denotes the (negative) framing of the $i$-th vertex of $\Gamma$. Assuming that the distinguished leaf $S_0$ is attached to the vertex labelled by $1$, we prove in Subsection \ref{subsection:Alexander} that the value of $\mathcal F$ changes only under $1$-steps. This observation leads to our invariant $\height[V]$ of a full path $[V]$, see Definition \ref{def}; it measures the maximal difference of the values of $\mathcal F$ attained along the full path, or equivalently the maximal number of consecutive $1$-steps that can be taken inside $[V]$.

Following the convention of \cite{BP}, we call $(m(1)+2,\ldots,m(|\Gamma|)+2)\in\text{Char}(\Gamma)$ the \emph{canonical vector}; see also \cite{Nemethi}. The significance of the height is that it admits a direct contact-geometric interpretation. Namely, if $V_{\text{can}}$ denotes the canonical vector, then the maximal twisting number is completely determined by the height of the full path containing $V_{\text{can}}$.

\begin{teo}\label{cor:tw}
If $M=M(e_0;r_1,\dots,r_n)$ with $n\geq2$ has a negative-definite graph, then all negative-twisting tight contact structures $\xi$ on $M$ have the same twisting number (of the corresponding regular fibre $K$): it is equal to  \[\emph{tw}(M,\xi,K)=-1-\height[V_{\emph{can}}]= \left\{\begin{aligned} 
&-1  \hspace{1.03cm}\text{ when }e_0<-1\\
&-1-\#  \hspace{0.24cm}\text{ when }e_0=-1 \end{aligned}\right.\]
where $\#$ is the integer given in Equation \eqref{eq:sharp}.
\end{teo}

Apart from $L$-spaces, for general Seifert fibred spaces only some bounds on the twisting number have previously appeared in literature. For instance, Mark and Tosun claimed in \cite[Theorem 4.6]{MT} that $\text{tw}(\Sigma(a_1,...,a_n),\xi)>-\sqrt{a_1\cdots a_n}$ for any tight structure $\xi$ on a canonically oriented Brieskorn sphere.

\medskip
\subsection*{A bound for the maximal \texorpdfstring{$\mathbf{tb}$-number}{}} The proof of the claim in Theorem \ref{teo:classification} that non-vanishing of $c^+(\xi)$ implies that the maximal twisting number of $\xi$ is negative follows from this chain of inequalities:
\[\text{tw}(M,\xi,K)+\dfrac{1}{-e(M)}=\text{TB}_\xi(K)\leq\tau_\xi(K)+\tau_{\overline\xi}(K)-1<\dfrac{1}{-e(M)}\:,\]
where the second inequality is seen from the interpretation of $HF^-(M)$ through the full path algorithm of Ozsv\'ath and Szab\'o \cite{OSz-fullpath}, while the first one comes from Hedden's $\tau_\xi$-inequality established in \cite{Hedden}, applied to both $K$ and $-K$; in addition, the fraction is the fibration framing of the knot $K$. 

In order to show that $\tau_\xi(K)+\tau_{\overline\xi}(K)-1$ is an upper bound for the maximal $\tb$-number of $K$, we need the following refined\footnote{The proof of \cite[Theorem 1.2]{Hedden} given by Hedden works for the self-linking number but not for the $\tb$-number.} version of \cite[Theorem 1.2]{Hedden} which takes into account the action of the involution $\J$ on the knot Floer homology groups, when reversing the orientation of $K$, see Section \ref{section:three}.

\medskip
\begin{teo}
 \label{teo:inequality1}
 If $L\subset(Y,\xi)$ is a link in a rational homology contact three-sphere such that $\widehat c(\xi)\neq0$ then \[\tb_\xi(\mathcal L)\leq\emph{TB}_\xi(L)\leq\tau_\xi(L)+\tau_{\overline \xi}(L)-|L|\] for any Legendrian realisation $\mathcal L$ of $L$, where $\tau_\xi(L):=-\tau_{\widehat c(\xi)}(L)$. 
\end{teo}

\subsection*{Obstructing contact-type embeddings of Brieskorn spheres} 
The classification theorem also has consequences for contact-type embedding problems. Using Theorem \ref{teo:classification} we can guarantee that a $\xi$ with self-conjugate contact invariant is itself self-conjugate; that is, $\xi$ is contact isotopic to its conjugate $\overline\xi$. 
Our obstruction then relies on proving Proposition \ref{prop:main}, which shows that if $\xi$ is a self-conjugate structure  among the ones in Theorem \ref{teo:classification} then $(M,\s_\xi)$ cannot have vanishing correction term, and this puts a strong restriction on the structure of the Heegaard Floer groups of the 3-manifold. 

\begin{prop}
 \label{prop:main} 
 Suppose that $M$ is a negative-definite Seifert fibred space different from $S^3$, and let $\s_0$ be a spin structure on $M$. Then every fillable contact structure $\xi$ on $M$ such that $\s_\xi=\s_0$ has either $c^+(\xi)\neq\Theta^+_{\s_0}$ or the Heegaard Floer correction term $d:=d(M,\s_0)$ is positive.
\end{prop}

We recall that $\Theta^+_{\s_0}\in HF^+(-M,\s_0)$ is the only non-zero element in $\Ker U\:\cap U^n\cdot HF^+(-M,\s_0)$ for every $n\geq0$. A weaker version of Proposition \ref{prop:main} was given by Karakurt in \cite[Theorem 1.4]{Karakurt} as an obstruction to the planarity of $\xi$. 

We say that a contact 3-manifold $(Y,\xi)$ admits a \emph{separating} contact-type embedding in a symplectic filling $(X,\omega)$ when $\omega\lvert_{\nu(Y)}$ coincides with the symplectisation of $(Y,\xi)$, and $Y$ disconnects $X$ into two pieces none of them being a symplectic cap of $(Y,\xi)$. In \cite{MT,MT3} Mark and Tosun study the problem of whether a Brieskorn sphere admits a separating contact-type embedding in a symplectic 4-manifold. 
Applying Proposition \ref{prop:main}, together with the obstruction from \cite[Theorem 2.1]{MT}, we obtain the following result, which in particular provides an alternative complete proof of \cite[Theorem 1.1]{MT}.

\begin{cor}
 \label{gompf} 
 No canonically oriented Brieskorn sphere $\Sigma(a_1,...,a_n)$ admits a separating contact-type embedding in any symplectic filling of a lens space. In particular, they do not arise as the boundary of any rationally convex domain in $\C^2$.
\end{cor}

Proposition \ref{prop:main} extends beyond Brieskorn spheres and yields a general Heegaard Floer obstruction to separating contact-type embeddings. 
Let $\F$ be the field with two elements. In the next statement we use the original definition of $L$-space given by Ozsv\'ath and Szab\'o; namely, we say that a 3-manifold $Y$ is a \emph{generalised $L$-space} if $HF_{\text{red}}(Y)=0$, or equivalently that the dimension of $\widehat{HF}(Y,\s)$ as an $\F$-vector space is one if $\s$ is torsion, and zero otherwise.

\begin{teo}
 \label{prop:convex}
 Suppose that $M$ is a negative-definite Seifert fibred space. 
 If $(M,\xi)$ admits a separating contact-type embedding in a strong symplectic filling $(X,\omega)$ of a generalised $L$-space, then $\dim_{\F}\widehat{HF}_d(M,\s_\xi)=1$. In particular, this happens when $(M,\xi)$ is the boundary of a rationally convex domain in a closed K\"ahler four-manifold $(\widehat X,\omega,J)$.
\end{teo}

Theorem \ref{prop:convex} holds for example in the case of $\widehat X=\C P^2$ with its unique symplectic structure. This gives evidence for a positive answer to \cite[Question 1.10]{MT3} when restricting to negative-definite Seifert fibred spaces.

\subsection*{Acknowledgements} {\smaller[1] We are grateful to Antonio Alfieri for all his comments and related discussions. We wish to thank: Matthew Hedden for a clarification about the $\tau$-Bennequin inequality, Francesco Lin for sharing his expertise in different Floer homology theories, and Ian Zemke for the useful conversations during the early stages of the project. We thank the Matematiska institutionen at Uppsala universitet for their friendly hospitality. A.C. has been partially supported by the HORIZON-ERC-2023-ADG 101141468 KnotSurf4d project.}

\section{The full path and the height function}
\label{section:two}
\subsection{The full path algorithm}
Throughout the paper we assume that the reader is familiar with the basics of Heegaard Floer homology \cite{OSz-involution}. Since we are mainly interested in Seifert fibred spaces,  we only work with the Heegaard Floer homology of graph manifolds. The Heegaard Floer groups of these 3-manifolds have been studied in depth \cite{OSz-fullpath, Nemethi,DM,Lattice2}, and were fruitfully applied in 3-dimensional contact  topology \cite{LS-existance,Karakurt,Irena,BP}. 

We now review the full path algorithm of Ozsv\' ath and Szab\' o \cite{OSz-fullpath}, on which our results mainly rely. We use the same  terminology, and conventions as in their paper.

Let $\Gamma$ be a negative-definite tree which can be turned into an $L$-space by lowering the framing of one vertex; such a graph is called \emph{almost-rational}, see \cite{Nemethi0,Nemethi}.
We write $P=X(\Gamma)$ for the corresponding plumbing 4-manifold, and $Y$ for its boundary rational homology sphere. Since $P$ is simply connected, one can identify a $\Spin^c$-structure $\mathfrak u\in\Spin^c(P)$ with its first Chern class $c_1(\mathfrak u)\in  H^2(P;\Z)$. Furthermore, since the spheres $S_1, \dots, S_{|\Gamma|}$ associated to the vertices of the plumbing $\Gamma$ form a basis of $H_2(P; \Z)$, we have a natural identification $H^2(P;\Z)\simeq \Z^{|\Gamma|}$. This allows us to represent the Chern class $c_1(\mathfrak u)$ of a $\Spin^c$-structure $\mathfrak u$ on the plumbing with a vector $V=(v_1,\dots , v_{|\Gamma|})$, where 
$v_i=\langle c_1(\mathfrak u) , [S_i] \rangle$.
Any vector $V\in \Z^{|\Gamma|}$ of this kind is called \emph{characteristic vector}, and is characterised by the property that  $v_i\equiv m(i) \text{ mod }2$, where $m(i)$ denotes the framing of the $i$-th vertex of $\Gamma$.

Since $P$ is a negative-definite 4-manifold with $b_1(P)=0$, it follows from the properties of the cobordism maps \cite[Proposition 9.4]{OSz-negative} that $F^-_{P,\mathfrak u}:HF^-(S^3)\rightarrow HF^-(Y)$ gives rise to a non-torsion element $F^-_{P,\mathfrak u}(1)\in HF^-(Y)$. In \cite{OSz-fullpath} Ozsv\'ath and Szab\'o show that each element $F^-_{P,\mathfrak u}(1)$ is identified with the equivalence class $[V]$, of the characteristic vector $V$ associated to $\mathfrak u$, determined by the \emph{full path} algorithm, which we now describe.

In what follows we denote by $Q$ the intersection form of $P$. In the basis  $\{[S_1],...,[S_{|\Gamma|}]\}$ one has that $Q_{ij}= [S_i]\cdot[S_j]=|S_i\pitchfork S_j|$. Rephrasing  \cite{OSz-fullpath} we say that a characteristic vector $V$: 
\begin{itemize}
 \item \emph{initiates} its full path if $m(i)+2\leq v_i\leq-m(i)$ for each $i=1,...,|\Gamma|$;
 \item \emph{terminates} its full path if $m(i)\leq v_i\leq-m(i)-2$ for each $i=1,...,|\Gamma|$;
 \item originates a \emph{step}  to $V'=V+2\text{PD}[S_j]$ if $v_j=-m(j)$ for some $j$;
 \item makes its full path \emph{drop out} if $|v_i|>-m(i)$ for some $i$.
\end{itemize}
Two vectors are in the same full path if there is a finite sequence of steps from one to the other.
When a full path contains both an initial and a terminal vector, and it does not drop out, we say that it \emph{ends correctly}.  

\begin{teo}[Ozsv\'ath-Szab\'o, N\'emethi]
   \label{teo:correction_term} 
 Let $Y$ be a three-manifold presented by an almost-rational graph. Then the correction term $d:=d(Y,\s)$ is the maximal Maslov grading of a non-zero homogeneous element in $HF^-(Y,\s)$. Furthermore, the equivalence classes of vectors (the full paths) generate $HF^-(Y)$ as an $\F[U]$-module; among these, the full paths that end correctly form a basis of $HF^-(Y)/U\cdot HF^-(Y)$ as an $\F$-vector space.
\end{teo}
\begin{proof}
 Ozsv\'ath and Szab\'o \cite[Theorem 1.2 and Proposition 3.2]{OSz-fullpath} and N\'emethi \cite[Theorem 8.3]{Nemethi0} show that the kernel of the $U$-multiplication $\Ker U\subset HF^+(-Y)$ is dual to the vector space spanned by the full paths that end correctly. The fact that $\Ker U\simeq HF^-(Y)/U\cdot HF^-(Y)$ follows from the canonical duality $HF^+(-Y)\simeq HF^-(Y)^\bullet$ established in \cite[Proposition 2.5]{OSz-involution}, see also Subsection \ref{subsection:duality} below. 
\end{proof}

Take $\{[D_1],...,[D_{|\Gamma|}]\}$ the basis of $H_2(P,Y;\Z)$ given by the duals of the spheres $S_i$'s, and consider $\{e_1,...,e_{|\Gamma|}\}$ the canonical basis of $H^2(P;\Z)\simeq\Z^{|\Gamma|}$, identified through Poincar\'e duality with the previous one. Fix $\mathfrak u\in\Spin^c(P)$, with respect to these bases one has $c_1^2(\mathfrak u)[P]:=V^TQ^{-1}V$ for the characteristic vector $V$ corresponding to $c_1(\mathfrak u)$. Then by \cite[Equation (3)]{OSz-fullpath} it follows that the Maslov grading of $V$ is given by $M(V)=\frac{V^TQ^{-1}V+|\Gamma|}{4}$.

Let us recall that $\widehat{HF}^\text{ev}(Y,\s)\subset\widehat{HF}(Y,\s)$ is the subspace generated by the homology classes $\alpha$ such that $M(\alpha)-d$ is even.
 
\begin{prop}
 \label{prop:correction_term}
 Let $\psi_*$ be the map induced in homology by the projection of $CF^-(Y,\s)$ over the quotient complex $\widehat{CF}(Y,\s):=CF^-(Y,\s)/\:U=0$. When restricting to $\widehat{HF}^\emph{ev}(Y,\s)$ the map $HF^-_*(Y,\s)/U\cdot HF^-(Y)\rightarrow\widehat{HF}_*^\emph{ev}(Y,\s)$ is an isomorphism of $\F$-vector spaces. Moreover, the correction term $d$ is the maximal Maslov grading of a non-zero homogeneous element in $\widehat{HF}(Y,\s)$.
\end{prop}
\begin{proof}
 As we mentioned above the homology classes corresponding to characteristic vectors are non-torsion with respect to the $\F[U]$-module structure; hence, by Theorem \ref{teo:correction_term} we see that $HF^-(Y,\s)$ is supported in Maslov gradings with the same parity of $d$. From the first homomorphism theorem, we just need to show that $\Imm\psi_*=\widehat{HF}^\text{ev}_*(Y,\s)$; this and the second claim in the statement both follow from this fact: every non-zero class $\alpha\in\widehat{HF}(Y,\s)$ such that either $M(\alpha)\geq d$ or $\alpha\in\widehat{HF}^\text{ev}(Y,\s)$ is in $\Imm\psi_*$.

 We prove the latter claim. Suppose that $\widehat{HF}(Y,\s)\ni\alpha=[a]\notin\Imm\psi_*$, then from standard Heegaard Floer properties, see \cite[Lemma 2.1]{AC} for a detailed discussion, there is a non-zero class $\beta=[b]\in HF^-(Y,\s)$ such that $\partial^-a=U^nb$ for some $n\geq1$; hence, we have $M(\alpha)-1=M(\beta)-2n$ which implies $M(\beta)=M(\alpha)+2n-1$. Since $\beta$ has the same parity of $d$, the same cannot hold true for $\alpha$; moreover, if $M(\alpha)\geq d$ then $M(\beta)>d$, but this contradicts the maximality of $d$ in $HF^-(Y,\s)$.
\end{proof}

To summarise, if $Y$  can be represented by an almost-rational graph (for example a Brieskorn sphere, or more generally a Seifert fibred space) then one can find a canonical basis for the Heegaard Floer homology group $\widehat{HF}^\text{ev}(Y,\s)$. Indeed, one can search for a collection of characteristic vectors $\{V_1,...,V_t\}$ restricting to the $\Spin^c$-structure $\s$, and find all possible full paths using the algorithm of Ozsv\' ath and Szab\' o. It follows from Theorem \ref{teo:correction_term} and Proposition \ref{prop:correction_term} above that $\{\gamma_i=\psi_*[V_i]: i=1, \dots , t \}$ forms a basis of $\widehat{HF}^\text{ev}(Y,\s)$ as an $\F$-vector space.

\subsection{The conjugation involution}
In what follows we exploit a certain symmetry of the Heegaard Floer chain complex. We recall that there is a map $\mathcal J: HF^\circ(Y) \to HF^\circ(Y)$ obtained by canonically identifying the intersection points of a given Heegaard diagram $(\Sigma, \alpha, \beta)$ of $Y$ to the ones of the conjugate diagram $(-\Sigma,  \beta, \alpha)$, where the orientation of the Heegaard surface is flipped, and the role of the $\alpha$- and the $\beta$-curves swapped \cite[Theorem 2.4]{OSz-involution}. Note that $\mathcal J$ preserves the Maslov grading, squares to the identity, and maps $HF^\circ(Y, \s)$ to  $HF^\circ(Y, \overline{\s})$.

It was noted by Hendricks and Manolescu \cite{HM} that, composing the latter identification with the chain map induced by a sequence of Heegaard moves connecting the two Heegaard diagrams, one obtains an involution denoted by $\iota$ on the chain complex itself. They showed that $\iota$ induces $\J$ in homology; in other words, one has $\iota_*=\J$. In this paper we work with the involution $\J$, as we are only interested in the behaviour at the homology level.   

One can compute $HF^-(Y,\s) $ using the Ozsv\' ath and Szab\' o full path algorithm; indeed, in this set-up the conjugation involution $\mathcal J$ acts on full paths by negating characteristic vectors, that is $\mathcal{J}[V]=[-V]$ for any $V$, as shown by Dai and Manolescu \cite[Theorem 3.1]{DM}, for manifolds presented by an almost-rational graph.
The involution $\mathcal J$ was instrumental in Ghiggini's identification of strongly fillable contact manifolds without Stein fillings \cite{G-notStein}. 

\subsection{The Alexander filtration of a leaf}
\label{subsection:Alexander}
We consider links that can be represented by meridians in a surgery description of a graph manifolds as shown in Figure \ref{graph} (right). These links are represented by some unlabelled vertices of the plumbing tree as in Figure \ref{graph} (left).  Links of this kind are called \emph{graph links} in some literature, and are the generalisation of algebraic links from singularity theory. We use the convention that vertices are enumerated so that the ones linked to the unlabelled vertices (which represent a graph link) appear before others.

\begin{figure}[ht] \begin{tikzpicture}[scale=0.6]
    \tkzDefPoints{0/1/A, 2/2/B, 2/0/C, -2/1/D} \tkzDefPoints{-1.8/0.9/X, -1.8/1.1/Y}
    \tkzDrawSegment(A,B)\tkzDrawSegment(A,C)\tkzDrawSegment(A,D)\tkzDrawSegment(D,X)\tkzDrawSegment(Y,D)
    \tkzDrawPoints[fill,black,size=5](A,B,C)
     \tkzLabelPoint[below left](A){$-1$} \tkzLabelPoint[right](B){$-2$}\tkzLabelPoint[right](C){$-3$} 
     \tkzLabelPoint[above](D){$K$}     
\end{tikzpicture}
\hspace{2cm} 
  \def\svgwidth{0.25\textwidth}
\begingroup%
  \makeatletter%
  \providecommand\color[2][]{%
    \errmessage{(Inkscape) Color is used for the text in Inkscape, but the package 'color.sty' is not loaded}%
    \renewcommand\color[2][]{}%
  }%
  \providecommand\transparent[1]{%
    \errmessage{(Inkscape) Transparency is used (non-zero) for the text in Inkscape, but the package 'transparent.sty' is not loaded}%
    \renewcommand\transparent[1]{}%
  }%
  \providecommand\rotatebox[2]{#2}%
  \newcommand*\fsize{\dimexpr\f@size pt\relax}%
  \newcommand*\lineheight[1]{\fontsize{\fsize}{#1\fsize}\selectfont}%
  \ifx\svgwidth\undefined%
    \setlength{\unitlength}{298.2961964bp}%
    \ifx\svgscale\undefined%
      \relax%
    \else%
      \setlength{\unitlength}{\unitlength * \real{\svgscale}}%
    \fi%
  \else%
    \setlength{\unitlength}{\svgwidth}%
  \fi%
  \global\let\svgwidth\undefined%
  \global\let\svgscale\undefined%
  \makeatother%
  \begin{picture}(1,0.73749185)%
    \lineheight{1}%
    \setlength\tabcolsep{0pt}%
    \put(0,0){\includegraphics[width=\unitlength,page=1]{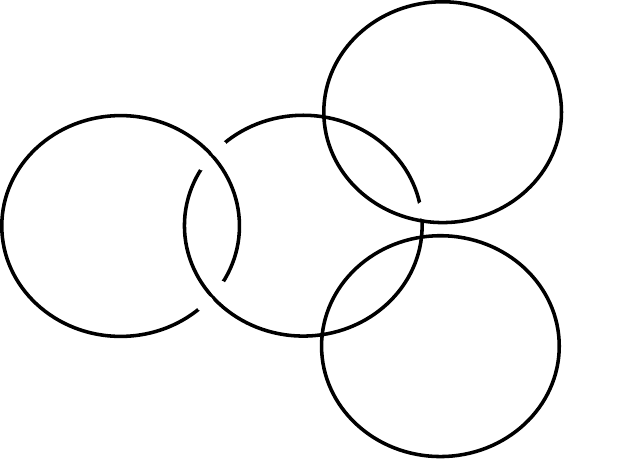}}%
    \put(0.00343043,0.5690166){\color[rgb]{0,0,0}\makebox(0,0)[lt]{\lineheight{1.25}\smash{\begin{tabular}[t]{l}$K$\end{tabular}}}}%
    \put(0.3608733,0.56984013){\color[rgb]{0,0,0}\makebox(0,0)[lt]{\lineheight{1.25}\smash{\begin{tabular}[t]{l}$-1$\end{tabular}}}}%
    \put(0.92268888,0.65131122){\color[rgb]{0,0,0}\makebox(0,0)[lt]{\lineheight{1.25}\smash{\begin{tabular}[t]{l}$-2$\end{tabular}}}}%
    \put(0.92148384,0.11379183){\color[rgb]{0,0,0}\makebox(0,0)[lt]{\lineheight{1.25}\smash{\begin{tabular}[t]{l}$-3$\end{tabular}}}}%
    \put(0,0){\includegraphics[width=\unitlength,page=2]{Graph.pdf}}%
  \end{picture}%
\endgroup%

     \caption{\smaller[1]{A negative-definite tree realising the positive trefoil in $S^3$.}}
     \label{graph}
\end{figure}     

Throughout the paper we assume that we have a knot $K$ as a leaf of the first vertex. Given a plumbing $P=X(\Gamma)$, and a graph knot $K\subset \partial P$ one can find  a properly embedded disk $D\subset P$ bounded by $K$, and intersecting transversely in one point the sphere $S_1$ associated to the (first) vertex of $\Gamma$ which it is attached to. 
Take $\{[D_1],...,[D_{|\Gamma|}]\}$ the basis of $H_2(P,Y;\Z)$ given by the duals of the spheres $S_i$'s: we write $[D]\cdot[D]$ for its self-intersection, determined by the rational number $[D_i]\cdot[D_j]:=|D_i\pitchfork\widehat D_j|=Q^{-1}_{ij}$ for any $i,j$ where $\widehat D_j$ is obtained by capping off $D_j$ with a (possibly rational) Seifert surface; moreover, fixed $\mathfrak u\in\Spin^c(P)$, we write $c_1(\mathfrak u)[D]$ for the evaluation of its relative homology class. Take $\{e_1,...,e_{|\Gamma|}\}$ the canonical basis of $\Z^{|\Gamma|}$, identified through Poincar\'e duality with the one of $H_2(P,Y;\Z)$ given by the $[D_i]$'s, and $V$ the characteristic vector corresponding to $c_1(\mathfrak u)$, we set $c_1(\mathfrak u)[D]:=[D]\cdot c_1(\mathfrak u)=[S_1]\cdot\text{PD}(c_1(\mathfrak u))=e_1^TQ^{-1}V$.
Following the convention in \cite{Alfieri,Lattice2,Lattice1} we  define the \emph{Alexander filtration} of $K$: 
\[\mathcal F(V)=-\dfrac{1}{2}\big([D]\cdot[D]-c_1(\mathfrak u)[D]\big)=\dfrac{1}{2}\big(-e_1^TQ^{-1}e_1+e_1^TQ^{-1}V\big)\:.\] 
In the subsequent sections  we write $V\cdot W$ in place of $V^TQ^{-1}W$ for the corresponding rational bilinear form on $\Z^{|\Gamma|}\simeq H^2(P;\Z)\simeq H_2(P,Y;\Z)$. We shall need the following two lemmas.

\begin{lemma}
\label{lemA}
 Let $Y$ be the three-manifold represented by an almost-rational graph $\Gamma$. Then the Alexander filtration is monotone along each full path, that is if $V^1, \dots, V^q$ are consecutive steps in a full path then $\mathcal{F}(V^1) \leq \dots \leq \mathcal{F}(V^q)$.
 More specifically, we have that \[\mathcal{F}(V^i) \leq \mathcal{F}(V^{i+1})  \leq \mathcal{F}(V^i)+1\] for  $i=1,...,q-1$.
\end{lemma}
\begin{proof} 
 Consider a step in a full path from $V$ to $W$; hence, by definition $W=V+2\text{PD}[S_i]=V+2Qe_i$, where $i$ is such that $v_i=-m(i)$. It follows that \[\begin{aligned}
    \mathcal F(W)-\mathcal F(V)&=\dfrac{1}{2}\left(e_1\cdot W-e_1\cdot V\right)=\\ &=\dfrac{1}{2}\left[e_1\cdot(V+2Qe_i)-e_1\cdot V\right]=\dfrac{1}{2}(2e_1^Te_i)=\left\{\begin{aligned}
        &1\hspace{1cm}\text{ if }i=1 \\ &0\hspace{1cm}\text{ if }i\neq1 
    \end{aligned}\right. \:.
 \end{aligned}\] 
 Therefore, we obtain $0\leq\mathcal F(W)-\mathcal F(V)\leq1$ which implies $\mathcal F(V)\leq\mathcal F(W)\leq\mathcal F(V)+1$.
\end{proof} 

It follows from Lemma \ref{lemA} that if $(Y,\s_0)$ is an integral homology sphere then for each full path $[W]$ there is an interval $[a_{[W]},b_{[W]}]$ such that every integer in it is realised by the Alexander filtration of some vector in $[W]$.

\begin{lemma}
\label{lemB} 
 Let $Y$ be the three-manifold represented by $\Gamma$, and assume the latter is an almost-rational graph where $K$ is linked to a bad vertex (if there is one). Say $V$ and $W$ are characteristic vectors that restrict to the same $\Spin^c$-structure on $Y$ and such that $\mathcal{F}(V) = \mathcal{F}(W)$; if their full paths $[V]$ and $[W]$ end correctly then $[V]=[W]$.
\end{lemma}
\begin{proof} 
 Suppose that there are $V$ and $W\in\Z^{|\Gamma|}$ as in the statement. Then  $e_1\cdot(V-W)=0$ which yields to \[V-W\in (Q^{-1}e_1)^\perp=\text{Span}\{2Qe_2,...,2Qe_{|\Gamma|}\}\:.\]
 This is because $\dim(Q^{-1}e_1)^\perp=|\Gamma|-1$ and $V-W$ has even coordinates. Since $V$ and $W$ induce the same $\Spin^c$-structure on the boundary if and only if there exists a unique $A\in\Z^{|\Gamma|}$ such that $2QA=V-W$, we can now write \begin{equation} V=W+2\lambda_2Qe_2+...+2\lambda_{|\Gamma|}Qe_{|\Gamma|}
 \label{eq:lambda}
 \end{equation} where each $\lambda_i$ is an integer. We wish to show that $V$ and $W$ are connected by steps. Therefore, up to going further along the full path, but not crossing any central step, we may assume that $m(i)\leq v_i<-m(i)$ and $m(j)\leq w_j<-m(j)$ for every $i,j=2,...,|\Gamma|$. 
   
 Denote by $V'$ and $W'$ the vectors obtained by removing the first coordinate to $V$ and $W$. Since $V$ and $W$ end correctly the same is necessarily true for $V'$ and $W'$; moreover, we have that \[V'-W'=2\lambda_2Q'e_1+...+2\lambda_{|\Gamma|}Q'e_{|\Gamma|-1}\] where $Q'$ is the matrix gotten by removing the first row and column from $Q$. This clearly implies that the vector $B=(\lambda_2,...,\lambda_{|\Gamma|})$ satisfies $2Q'B=V'-W'$, and then $V'$ and $W'$ induce the same $\Spin^c$-structure $\s'$ on $Y'$, the manifold presented by the graph $\Gamma$ without the first vertex.

 Since $\Gamma$ is almost-rational, lowering the framing of the first vertex changes it to a rational graph, that is a negative-definite plumbing tree presenting an $L$-space \cite{Nemethi}. It is known that a subgraph of a rational graph is also rational; hence, we have that $Y'$ is itself an $L$-space. We conclude that $[V']=[W']$ because there is only one generator in $\widehat{HF}(Y',\s')$. In addition, we are assuming that $V'$ and $W'$ never reach the maximal coordinate, and then they are terminal vectors in their full path. It is proved in \cite[Proposition 3.2]{OSz-fullpath} that initial and terminal vectors are unique; therefore, we proved that $V'=W'$.

 In conclusion, since $Q'$ has non-zero determinant, we can now argue that $\lambda_2=\cdots=\lambda_{|\Gamma|}=0$ and then $V=W$.
\end{proof} 

The following is an immediate consequence of Lemmas \ref{lemA} and \ref{lemB}, and the uniqueness of initial and terminal vectors \cite[Proposition 3.2]{OSz-fullpath}. 
\begin{cor}
 Suppose that $Y$ is an integral homology sphere as in Lemma \ref{lemB}.
 Then for every full path $[V_i]$ which ends correctly there is an interval $[a_i,b_i]$ such that every integer in it is realised by the Alexander filtration of some vector in $[V_i]$. Running through all such full paths $[V_1],...,[V_t]$, the intervals are ordered so that $a_1\leq b_1<a_2\leq\:...\leq b_{t-1}<a_t\leq b_t$.
\end{cor}
Finally, we introduce the height function that we will use in Sections \ref{section:four} and  \ref{section:five}. 
\begin{defin}[Height]
 \label{def}
       For any characteristic vector $W$ whose full path ends correctly, assume that $V$ is the initial vector of its full path $[W]$ while $V'$ is the initial vector of the full path of its conjugate $[-W]$; then we write the \emph{height} of the full path as \begin{equation}
\height[W]=\max_{Z\in[W]}\mathcal F(Z)-\mathcal F(V)=\mathcal F(-V')-\mathcal F(V)=-e_1\cdot\left(\dfrac{V+V'}{2}\right)
\label{eq:height}
\end{equation} as the difference between the maximal and the minimal value of the Alexander filtration, induced by the leaf $K$, on a given full path class.
\end{defin}
Note that $\height[W]$ is equal to the number of \emph{central steps}, that are steps from $W$ to $W+2Qe_1$, from the initial to the terminal vector in the full path $[W]$. 

\subsection{Duality}
\label{subsection:duality}
It is standard in Heegaard Floer theory to identify canonically $HF^+(-Y)$ with $HF^-(Y)^\bullet$, see \cite{AC} for details. Given $[V_1],...,[V_t]$ the homology classes corresponding to vectors whose full path ends correctly, then Ozsv\'ath and Szab\'o \cite{OSz-fullpath} define elements in $HF^+(-Y)$, when $Y$ is presented by an almost-rational graph, as functionals $T_{[V]}:HF^-(Y)\rightarrow\F$ such that $T_{[V]}[V_i]=1$ if $[V_i]=[V]$, and zero otherwise. The full path gives an equivalence relation also on these functionals in the sense that $T_{[V]}$ and $T_{[W]}$ are the same class in $HF^+(-Y)$ if and only if $[V]=[W]$. The following is one of the main results in \cite{OSz-fullpath}.
\begin{prop}[Ozsv\'ath-Szab\'o]
 The functionals $\{T_{[V_1]},...,T_{[V_t]}\}$ form a basis of $\Ker U\subset HF^+(-Y)$; moreover, the Maslov grading satisfies $M(T_{[V_i]})=-M(V_i)$ for any $i=1,...,t$.
\end{prop}
We have the following interpretation in terms of cobordism maps. Let $F^+_{\overline P,\mathfrak u}:HF^+(-Y)\rightarrow HF^+(S^3)$ be the map induced by the plumbing $P$ turned upside-down, and $\mathfrak u$ the $\Spin^c$-structure corresponding to $V$. Then we have \[\Ker F^+_{\overline P,\mathfrak u}=(\Imm F^-_{P,\mathfrak u})^\perp=[V]^\perp=\text{Span}\{T_{[W]}\:|\:[V]\neq[W]\}\:.\] Therefore, one has $F^+_{\overline P,\mathfrak u}(T_{[W]})=1$ if and only if $[W]=[V]$.
Say $d=d(Y,\s)$ is the Heegaard Floer correction term, and $\widehat{HF}^\text{ev}(-Y)\subset\widehat{HF}(-Y)$ consists of the classes whose Maslov grading has the same parity of $d$. We make the following observation.
\begin{prop}
 \label{prop:plus}
 The map $\rho^*:\widehat{HF}^\emph{ev}_*(-Y,\s)\rightarrow\Ker U\subset HF^+_*(-Y,\s)$ induced by the inclusion of $\widehat{CF}_*(-Y)$ into $CF^+_*(-Y)$ is an isomorphism of $\F$-vector spaces.
\end{prop}
\begin{proof}
  The map $\rho*$ is dual to $\psi_*:HF^-_*(Y,\s)\rightarrow\widehat{HF}^\text{ev}_*(Y,\s)$, and the latter is an isomorphism, when projecting over $\Ker\psi_*=U\cdot HF^-(Y,\s)$, because of Proposition \ref{prop:correction_term}.  
\end{proof}

\section{Tau-invariants and Bennequin inequalities}
\label{section:three}
Given a Seifert space $M=M(e_0;r_1,...,r_n)$ we can consider a regular fibre $K\subset M$. This gives a knot that can be represented by an unlabelled vertex attached to the central vertex of the standard graph $\Gamma$ of $M$.
We recall that in the case of the Brieskorn sphere $\Sigma(a_1,...,a_n)$ the self-intersection $[D]\cdot[D]$ of the disk $D\subset X(\Gamma)$ bounded by $K$ equals $-a_1\cdots a_n$.
For a general Seifert fibred space  one has that $[D]\cdot[D]=\frac{1}{e(M)}$, where  $e(M):=e_0+r_1+\cdots +r_n$. Note that $e(M)<0$ if and only if the standard graph of $M$ is negative-definite.

Take $\gamma\in\widehat{HF}(M,\s)$ such that $\psi_*[W]=\gamma$ for some full path $[W]$ ending correctly. The same arguments, as presented in \cite[Theorem 1.2]{Alfieri} and \cite[Theorem 7.3]{AC} for any almost-rational graph, show that the invariant $\tau_\gamma(K)$, associated to a homology class $\gamma$ \cite{AC}, can be expressed as follows when $M$ is Seifert fibred: 
\begin{equation}
   \label{eq:tau}
    \tau_\gamma(K)=\dfrac{1}{2}\left(\dfrac{1}{-e(M)}+\min_{[Z]=[W]}e_1\cdot Z\right)=\dfrac{1}{2}\left(\dfrac{1}{-e(M)}+e_1\cdot V_\gamma\right)
\end{equation} where $V_\gamma$ is the initial vector of the full path $[W]$ by Lemma \ref{lemB}.  
Then \cite[Theorem 1.2]{Alfieri} also shows that the invariant $\tau(Y,K,\s)$, the minimal among the $\tau$-invariants associated to a class mapped to $\Theta^+_\s$ 
\cite{GRS,Raoux}, can be computed in a similar way:
\[\tau(Y,K,\s)=\dfrac{1}{2}\left(\dfrac{1}{-e(M)}+\min_{\substack{M(Z)=d \\ Z\text{ restricts to }\s}}e_1\cdot Z\right)\] which again by Proposition \ref{prop:correction_term} and Lemma \ref{lemB} coincides with the minimum taken among all the initial vectors $V_\gamma$ with Maslov grading equal to $d=d(Y,\s)$.
The Maslov grading of $Z$ can be expressed as $M(Z)=\frac{Z\cdot Z+|\Gamma|}{4}$, see \cite[Equation (3)]{OSz-fullpath}.

We recall the usual $\tau$-Bennequin inequality for transverse links\footnote{Note that some authors may have different conventions for the orientation of the link, thus interchanging the role of $\tau_\xi$ and $\tau_{\overline\xi}$. For example compare Proposition \ref{prop:tau} with the similar inequalities in \cite{LW,Hedden,Raoux}.}. See  \cite[Theorem 1]{Olga}, for the original version in $S^3$, the proof of \cite[Theorem 1.2]{Hedden} for the general knots version, \cite[Theorem 1.1]{LW} about rationally null-homologous knots, and \cite[Theorem 5.6]{Cavallo} for the link case.
\begin{prop}[Plamenevskaya, Hedden, Li-Wu, Cavallo]
 \label{prop:tau}
 If $L\subset(Y,\xi)$ is a link in a rational homology contact three-sphere such that $\widehat c(\xi)\neq0$ then
  \[\emph{SL}_\xi(L)\leq2\tau_\xi(L)-|L|\:,\] where $\tau_\xi(L):=-\tau_{\widehat c(\xi)}(L)$. 
\end{prop}

We can find a similar upper bound for the $\tb$-number of Legendrian links. We write a complete proof since there is an imprecision in \cite[Theorem 1.2]{Hedden}; moreover, this bound has already been applied incorrectly\footnote{Mark and Tosun apply the incorrect version of \cite[Theorem 1.2]{Hedden} in the proof of \cite[Corollary 3.10]{MT}. This mistake is not meaningful because \cite[Corollary 2.9]{MT} holds for both $\tau_\xi(Y,K)$ and $\tau_\xi(-Y,K)$.}. The first statement is also needed in the proof of Theorems \ref{teo:classification} and \ref{teo:tw}.

For any contact structure $\xi$ on $Y$ we denote by $\overline\xi$ its conjugate structure, that is the contact structure obtained from $\xi$ by inverting the orientation of the planes. This equips the space of oriented contact structures with an involution, which is compatible with the conjugation on the induced $\Spin^c$-structures, so that $\s_{\overline\xi}=\overline{\s_{\xi}}$.
\begin{teo}
 \label{teo:inequality}
 If $L\subset(Y,\xi)$ is a link in a rational homology contact three-sphere such that $\widehat c(\xi)\neq0$ then \[\tb_\xi(\mathcal L)\leq\emph{TB}_\xi(L)\leq\tau_\xi(L)+\tau_{\overline \xi}(L)-|L|\] and \[\tb_\xi(\mathcal L)+|\rot_\xi(\mathcal L)|\leq2\max\{\tau_\xi(L),\tau_{\overline \xi}(L)\}-|L|\] for any Legendrian realisation of $L$. Consequently,  when $c^+(\xi)=\Theta^+_{\s_\xi}$ one has
  \[\emph{TB}_\xi(L)\leq\tau(Y,L,\s_\xi)+\tau(Y,L,\s_{\overline\xi})-|L|\:,\] which simplifies to $\emph{TB}_\xi(L)\leq2\tau(Y,L,\s_\xi)-|L|$ if $\s_\xi$ is spin.
\end{teo}
\begin{proof}[Proof of Theorems \ref{teo:inequality1} and \ref{teo:inequality}]
 The $\tau$-Bennequin inequality for $\mathcal L$ and its reverse are \[\tb_\xi(\pm\mathcal L)-\rot_\xi(\pm\mathcal L)\leq\text{SL}_\xi(\pm L)\leq2\tau_\xi(\pm L)-|L|\:,\] and we know that $\tb_\xi(-\mathcal L)=\tb_\xi(\mathcal L)$ and $\rot_\xi(-\mathcal L)=-\rot_\xi(\mathcal L)$; moreover, we have $\tau_\xi(-L)=\tau_{\overline\xi}(L)$ because there is a canonical identification $(\widehat{CF}(Y,\s_\xi),\mathcal F^{-L})\simeq(\widehat{CF}(Y,\overline\s_\xi),\mathcal F^L)$ of the link Floer homology filtred complexes, given by the conjugation $(-\Sigma,\beta,\alpha,\mathbf w,\mathbf z)$ of the Heegaard diagram $(\Sigma,\alpha,\beta,\mathbf w,\mathbf z)$ as in the definition of the involution $\mathcal J$, see \cite[Lemma 3.12]{OSz-multi}.
 
 For the second part of the statement, we write $\rho^*:\widehat{HF}(-Y)\rightarrow HF^+(-Y)$ for the map induced in homology by the inclusion as in Proposition \ref{prop:plus}. When $c^+(\xi)=\Theta^+_{\s_\xi}$ we have that 
 \[c^+(\overline\xi)=\rho^*(\widehat c(\overline\xi))=\rho^*(\J\widehat c(\xi))=\mathcal J\rho^*(\widehat c(\xi))=\J c^+(\xi)=\J \Theta^+_{\s_\xi}=\Theta^+_{\s_{\overline\xi}}\:,\]
 and we apply \cite[Lemma 2.8]{MT} for both $\xi$ and $\overline\xi$ to the first inequality. 
\end{proof}

\section{Evaluation of the twisting numbers} 
\label{section:four}
In order to prove our main results, we need to understand the \emph{maximal twisting number} of tight contact structures on negative-definite Seifert fibred spaces $M$. These are the integers that capture the maximal difference between the contact and fibration framing of a regular fibre $K$, that is \[\text{tw}(M,\xi,K):=\text{TB}_\xi(K)\:-\text{the fibration framing of } K=\text{TB}_\xi(K)+\frac{1}{e(M)}\:\] for a tight contact structure $\xi$ on $M$. By convention \cite{G-}, if $\text{tw}(M,\xi,K)\geq0$ then we say that $\xi$ is \emph{zero-twisting}; otherwise, we say it is \emph{negative-twisting}. When $n\geq3$ the Seifert fibration is determined by $M$, and we omit $K$; in particular, for the Brieskorn sphere $Y=\Sigma(a_1,\dots,a_n)$, we have $\text{tw}(Y,\xi)=\text{TB}_\xi(K)-a_1\cdots a_n$.

Our analysis of possible twisting numbers is based on the work of Ghiggini \cite{G-}; one could instead start from the analogous results of Massot \cite{Massot}. Extending arguments from \cite[Section 2]{G-} to the Seifert fibred spaces $M=M(e_0;r_1,\dots,r_n)$ with any $e_0\leq-1$ and any number $n\geq 2$ of singular fibres, we get that the existence of a tight contact structure with given twisting number is equivalent to the Seifert constants fulfilling particular numerical conditions.

\begin{prop} 
 \label{prop:Paolo}
 The Seifert fibred space $M=M(e_0;r_1,\dots,r_n)$ admits a tight contact structure $\xi$ with twisting number (of a regular fibre $K$) equal to  $\emph{tw}(M,\xi,K)=-q\leq -1$ if and only if for Seifert coefficients $r_i\in(0,1)\cap\Q$ there exist positive integer numbers $p_1,\dots,p_n$ such that 
\begin{itemize}
 \item $\frac{p_i}{q}>r_i$ and $\gcd(p_i,q)=1$;
 \item $p_1+...+p_n = -e_0q+n-2$;
 \item no rational number $\frac{k}{h}$ with $1\leq h< q$ and $1\leq k\leq p_i$ is in the interval $(r_i,\frac{p_i}{q})$.
\end{itemize}
\end{prop}
Note that, according to \cite[Lemma 2.2]{G-}, when $q>1$ also fractions of the form $\frac{k}{q}$ do not lie in the interval $(r_i,\frac{p_i}{q})$; in this case, the first and third conditions are equivalent to say that $\frac{p_i}{q}$ is a \emph{best upper approximation} for each $r_i$. However, when $q=1$ and $e_0<-2$ other integers might appear in the interval.

\begin{proof}[Proof of Proposition \ref{prop:Paolo}]
 The proof of the only if part is modelled on the corresponding proofs by Ghiggini \cite[Propositions 2.1 and 2.6]{G-}. Here, and then in Section \ref{section:five}, we adopt the convention that $-\partial(M\backslash V_i)$, for a tubular neighbourhood $V_i$ of the $i$-th singular fibre, is trivialised by the direction of a regular fibre $\infty$, and so that the meridian of $V_i$ has slope $-r_i$.

 By assumption we have a regular fibre $K$ with $\text{tw}(M,\xi,K)=-q\leq -1$, and we can isotope the fibration so that all the singular fibres $K_i$ are Legendrian as well. Up to stabilising $K_i$ (many times), we can assume that the slope on $-\partial(M\backslash \nu K_i)$, for $\nu K_i$ the standard neighbourhood of $K_i$, is arbitrary close to $-r_i$. Applying imbalance principle \cite[Proposition 3.17]{Honda1} to the vertical annuli between $K$ and a Legendrian ruling on $-\partial(M\backslash \nu K_i)$, we can then thicken $\nu K_i$ to neighbourhoods $N K_i$ so that the slope on $-\partial(M\backslash N K_i)$ equals $-\frac{p_i}{q}$ for some integers $p_i$. Since $-\partial(M\backslash N K_i)$ is obtained by attaching vertical bypasses to $-\partial(M\backslash \nu K_i)$, we have that $\frac{p_i}{q}>r_i$; moreover, since $-q$ is the maximal twisting number, we obtain $\gcd(p_i,q)=1$ and the third condition (as all rational numbers in the interval are realised by slopes on parallel tori in $N K_i\backslash \nu K_i$).
 In fact, by construction the positive integers are uniquely determined as $p_i=\lceil qr_i\rceil$, unless $\lceil qr_i\rceil= qr_i$ when we have $p_i=qr_i+1$ instead.

 Finally, for the second condition, we cut $M\backslash (\cup_iNK_i)$ along vertical annuli between $-\partial(M\backslash N K_i)$ and $-\partial(M\backslash N K_{i+1})$ for $i=1,\dots,n-2$, which after edge rounding \cite[Lemma 3.11]{Honda1} leaves us with a thickened torus with boundary slopes $-\frac{-p_1-\cdots-p_{n-1}+n-2}{q}+e_0$ and $-\frac{p_n}{q}$. When $q>1$, if the two numbers were not the same there would be again a slope with smaller denominator in between, contradicting maximality of $-q$. However, for $q=1$ the second condition is not enforced by contact topology; the positive integers $p_i$ (all equal to $1$) can be, and for $e_0<-2$ they indeed are, such that $-(-p_1-\cdots-p_{n-1}+n-2)+e_0$ is less than $-p_n$. However, in this case the thickened torus between the two slopes consists of $-e_0-1$ basic slices, which we add to $NK_n$, changing $p_n$ from $1$ to $-e_0-1$; a valid choice that fulfils all the conditions.

 For the if part we notice that Ghiggini in the proof of \cite[Theorem 4.5]{G-} constructs a transverse contact structure with twisting number $-q$, starting from the assumed numerical conditions on Seifert constants. All his arguments work the same with more singular fibres and other Euler numbers.
\end{proof}

Recall that $M$ has negative-definite standard graph if and only if $e(M):=e_0+r_1+\cdots +r_n<0$. We will see that negative-definiteness, together with the above numerical conditions, ensures that there is a unique possible twisting $-q$ for all tight structures on $M$, and we will determine its value. In the following, we order the singular fibres so that $r_1\geq r_2\geq \cdots \geq r_n$.

In the proof of the following proposition we will use the \emph{Farey sequences} to keep track of the rational numbers with small denominators. Here, we recall only the very basic facts and refer to \cite{NZM} for details. Farey sequences are sequences of rational numbers constructed from the starting sequence $\frac{0}{1},\frac{1}{0}$ by always taking the previous sequence and inserting $\frac{a+b}{c+d}$ between every pair of consecutive terms $\frac{a}{c},\frac{b}{d}$. We will need the following facts: every positive rational number appears as a term eventually, the fractions in each sequence come in increasing order, and that the consecutive fractions in any sequence form a basis of $\Z^2$ (when considered as vectors in $\Z^2$).

\begin{prop}
 \label{prop:tw}
 If $M=M(e_0;r_1,\dots,r_n)$ has negative-definite standard graph, then all negative-twisting tight contact structures on $M$ have the same  twisting number $-q$.
 Moreover, the integer $-q$ equals the twisting number of tight contact structures on $M(e_0;r_1,r_2)$.
\end{prop}

\begin{proof}
Let us first look at $M=M(e_0;r_1,\dots,r_n)$ with $e_0<-1$. Taking $q=1$, $p_i=1$ for $i=1,\dots,n-1$ and $p_n=-e_0-1$, the properties follow immediately. We just need to check that no $q>1$ gives a solution. We immediately notice that when $e_0\leq-n$ no such $q$ exists because, to fulfil the second condition, at least one $\frac{p_i}{q}$ is bigger than one, but then $1\in(r_i,\frac{p_i}{q})$ and $1<q$ fails the third condition. In the case that $e_0\in\{-n+1,\dots,-2\}$, we assume to the contrary that there exists $q>1$ and $p_i$ for $i=1,\dots,n,$ that satisfy the three conditions; 
in particular, one has \[\frac{p_1}{q}+\frac{p_2}{q}+\cdots+\frac{p_n}{q}=\frac{-e_0q+n-2}{q} \text{ and } \frac{p_i-1}{q-1}\leq r_i < \frac{p_i}{q}\] 
by the second and third condition respectively, and since $\frac{p_i-1}{q-1}< \frac{p_i}{q}$ when $\frac{p_i}{q}<1$.
Then \[r_1+\cdots+r_n\geq\frac{p_1-1}{q-1}+\frac{p_2-1}{q-1}+\cdots+\frac{p_n-1}{q-1}=\frac{-e_0q-2}{q-1}\geq -e_0\:,\]
which contradicts the negative-definiteness condition $e(M)<0$. In conclusion, when $e_0<-1$, the twisting number of any negative-twisting tight structure is $-1$.

The rest of the proof concerns $M=M(e_0;r_1,\dots,r_n)$ with $e_0=-1$. It relies on the presentation of rational numbers by Farey sequences. First, in the case that all $r_i<\frac{1}{2}$, we can finish very similarly to the case when $e_0=-2$ above; more precisely, we take $q=2$, $p_i=1$ for all $i=1,\dots,n$, and the properties follow immediately. We need to again check that no $q>2$ gives a solution. Such $q>2$ would satisfy 
 \[\frac{p_1}{q}+\frac{p_2}{q}+\cdots+\frac{p_n}{q}=\frac{q+n-2}{q} \text{ and } \frac{p_i-1}{q-2}\leq r_i < \frac{p_i}{q},\] 
by the second and third condition respectively, and since $\frac{p_i-1}{q-2}< \frac{p_i}{q}$ when $\frac{p_i}{q}<\frac{1}{2}$. But then \[r_1+\cdots+r_n\geq\frac{p_1-1}{q-2}+\frac{p_2-1}{q-2}+\cdots+\frac{p_n-1}{q-2}=\frac{q-2}{q-2}=1,\]
contradicting the negative-definiteness condition $e(M)<0$. In conclusion, we are done when all $r_i<\frac{1}{2}$; the twisting number equals $-2$. 

Otherwise, $r_1\geq \frac{1}{2}$ and we find the smallest $m\in\N$ such that $r_1<\frac{m-1}{m}$. If $r_2$, and thus all $r_i$ for $i>1$, are smaller than $\frac{1}{m}$, then $q=m,\ p_1=m-1,\ p_2=\cdots=p_n=1$ fulfil the conditions. Otherwise, since $r_1+\cdots+r_n<1$, we have $\frac{m-2}{m-1}\leq r_1 < 1- r_2 \leq\frac{m-1}{m}$, and since $\frac{m-2}{m-1}\text{ and }\frac{m-1}{m}$ are consecutive fractions in a Farey sequence, all rational numbers in that interval have denominator at least $2m-1$. If \[r_1< \dfrac{(m-2)+(m-1)}{(m-1)+m}=\dfrac{2m-3}{2m-1} < 1-r_2\:,\] then we take $\frac{p_1}{q}=\frac{2m-3}{2m-1}$, $\frac{p_2}{q}=\frac{2}{2m-1}$, and since \[1-r_1-r_2<\dfrac{m-1}{m}-\dfrac{m-2}{m-1}=\dfrac{1}{(m-1)m}\:,\] one has $r_i<\frac{1}{2m-1}=\frac{p_i}{q}$ whenever $i>1$. If not, we continue the same analysis with the sub-interval bounded by the consecutive fractions in the next level Farey sequence, so either $\left[\frac{m-2}{m-1},\frac{2m-3}{2m-1}\right]$ or $\left[\frac{2m-3}{2m-1},\frac{m-1}{m}\right]$, depending on which of the two contains rational numbers $r_1$ and $1- r_2$. Eventually, the integer $q$ will be the denominator of the mid point $\frac{a+b}{c+d}$ between two consecutive fractions $\frac{a}{c}, \frac{b}{d}$ in the Farey sequence, at the first level where 
\begin{equation}
    \label{eq:q}
    \frac{a}{c}\leq r_1<\frac{a+b}{c+d}< 1-r_2\leq\frac{b}{d}\:
\end{equation} then $q=c+d,\ p_1=a+b,\ p_2=c+d-a-b$ and $p_3=\cdots=p_n=1$.

We still need to prove that this is the only $q$ which satisfies the conditions. By construction, $q$ is the smallest integer that gives a solution, to prove that no larger $Q$ works, we once again obtain a similar contradiction to the case of $e_0=-2,\ q>1$ above. Let $Q>q$ satisfy 
 \[\frac{P_1}{Q}+\frac{P_2}{Q}+\cdots+\frac{P_n}{Q}=\frac{Q+n-2}{Q} \text{ and } \frac{P_i-p_i}{Q-q}\leq r_i < \frac{P_i}{Q},\] 
by the second and third condition respectively, and since $\frac{P_i-p_i}{Q-q}< \frac{P_i}{Q}$ when $\frac{P_i}{Q}<\frac{p_i}{q}$. 
Then we would have \[r_1+\cdots+r_n\geq\frac{P_1-p_1}{Q-q}+\frac{P_2-p_2}{Q-q}+\cdots+\frac{P_n-p_n}{Q-q}=\frac{(Q+n-2)-(q+n-2)}{Q-q}=1,\]
contradicting the negative-definiteness condition $e(M)<0$. Hence, there is no possible twisting number other than $-q$. Moreover, the value of $q$ is by Equation \eqref{eq:q} also in this case determined by $r_1$ and $r_2$ only.
\end{proof}

There is a simple lower bound for the twisting number that can be read directly from the proof of Proposition \ref{prop:tw}. The negative twisting number of tight structures on $M$ equals $-q=-1$ when $e_0<-1$, and when $e_0=-1$ it is bounded by $-q\geq -\min\{(1-r_1-r_2)^{-1},\:a_1+a_2\}$ where $a_i$ is the denominator of $r_i$ for $i=1,2$.
Of course, in any specific case the explicit value of the twisting number can also be determined; in particular, when $M(-1;r_1,r_2)=S^3$, we have $q=a_1+a_2$. In the following, we show how this value is expressed in the full path algorithm, and eventually compute it explicitly in general from negative continued fraction expansions of $r_1$ and $r_2$. We note that we could obtain the explicit value already by carefully looking into arithmetic of Farey sequences, but we wish to emphasise the topological reason behind the solution.

We denote by $Y_0=M(e_0;r_1,r_2)$ the lens space obtained by removing all the legs from the graph of $M$ except the two corresponding to the biggest coefficients; moreover, we write $Q_0$ for the intersection form of the standard graph of $Y_0$, which is naturally a negative-definite sub-graph of the one for $M$. We then have that $|\widehat{HF}(Y_0)|=|H_1(Y_0;\Z)|=\det(Q_0)$, and for any $\Spin^c$-structure $\s$ on $Y_0$ we write $W_{\s}$ for the initial characteristic vector of the full path $[W_{\s}]$, which is the unique generator of $\widehat{HF}(Y_0,\s)$.

\begin{teo}
 \label{teo:tw}
 Let $Y_0=M(e_0;r_1,r_2)$ be a Seifert fibred space whose standard graph is negative-definite, and $K$ its regular fibre. Then for any tight contact structure $\zeta$ we have that \[\emph{tw}(Y_0,\zeta,K)=-1-\height[W_{\s_{\emph{can}}}]=-1-\min_{\s\in\Spin^c(Y_0)}\left\{\height[W_\s]\right\}\:,\] where $\s_{\emph{can}}$ is the $\Spin^c$-structure induced by $\zeta_{\emph{can}}$ on $Y_0$, which is a lens space. Furthermore, the minimum of $\height[W_\s]$ is realised by any $\s$ induced by a tight contact structure.
\end{teo}

\begin{proof}
 We start by showing that $\text{tw}(Y_0,\zeta,K)=-1-\height[W_{\s_{\text{can}}}]$. We proved in Proposition \ref{prop:tw} that the twisting number does not depend on the tight structure $\zeta$; hence $\text{tw}(Y_0,\zeta,K)=\text{tw}(Y_0,\zeta_{\text{can}},K)$. Moreover, by definition of twisting number we have that $\text{tw}(Y_0,\zeta_{\text{can}},K)=\text{TB}_{\zeta_{\text{can}}}(K)+e_1^TQ_0^{-1}e_1$, where $e_1^TQ_0^{-1}e_1=[D]\cdot[D]$ is the self-intersection of the disk $D\subset X(\Gamma_0)$ bounded by $K$ (in other words, the negative of the fibration framing of $K$), and $\Gamma_0\subset\Gamma$ is the standard star-shaped sub-graph of $Y_0$.

 The knot $K$, a regular Seifert fibre of $Y_0$, has a Legendrian realisation $\mathcal K$ of $K$ which is the image of a standard torus knot in $(S^3,\xi_\text{std})$ under some contact $-1$-surgeries, describing the canonical Stein filling of the lens space $(Y_0,\zeta_{\text{can}})$, as in Figure \ref{resolution}; we can see this by performing a sequence of blow-downs on the standard graph of $Y_0$ until we end up with a minimal resolution graph, a technique described in \cite{BP,Nemethi}.
 In addition, a standard torus knot has a Lagrangian filling in $D^4$; by extending this over the contact surgeries we see that $\mathcal K$ also bounds a Lagrangian surface, and thus both $K$ and $-K$ are the boundary of a symplectic curve, in the canonical Stein filling of $(Y_0,\zeta_{\text{can}})$. 

 \begin{figure}[ht]  \begin{tikzpicture}[scale=0.6]
    \tkzDefPoints{0/1/A, 2/2/B, 2/0/C, -2/1/D} \tkzDefPoints{-1.8/0.9/X, -1.8/1.1/Y}
    \tkzDrawSegment(A,B)\tkzDrawSegment(A,C)\tkzDrawSegment(A,D) \tkzDrawSegment(D,X)\tkzDrawSegment(D,Y)
    \tkzDrawPoints[fill,black,size=5](A,B,C)
     \tkzLabelPoint[above left](A){$-1$} \tkzLabelPoint[right](B){$-3$}\tkzLabelPoint[right](C){$-5$} 
     \tkzLabelPoint[above left](D){$K$}     
 \end{tikzpicture}
 \hspace{2cm}
 \def\svgwidth{0.5\textwidth}
\begingroup%
  \makeatletter%
  \providecommand\color[2][]{%
    \errmessage{(Inkscape) Color is used for the text in Inkscape, but the package 'color.sty' is not loaded}%
    \renewcommand\color[2][]{}%
  }%
  \providecommand\transparent[1]{%
    \errmessage{(Inkscape) Transparency is used (non-zero) for the text in Inkscape, but the package 'transparent.sty' is not loaded}%
    \renewcommand\transparent[1]{}%
  }%
  \providecommand\rotatebox[2]{#2}%
  \newcommand*\fsize{\dimexpr\f@size pt\relax}%
  \newcommand*\lineheight[1]{\fontsize{\fsize}{#1\fsize}\selectfont}%
  \ifx\svgwidth\undefined%
    \setlength{\unitlength}{1215.07231945bp}%
    \ifx\svgscale\undefined%
      \relax%
    \else%
      \setlength{\unitlength}{\unitlength * \real{\svgscale}}%
    \fi%
  \else%
    \setlength{\unitlength}{\svgwidth}%
  \fi%
  \global\let\svgwidth\undefined%
  \global\let\svgscale\undefined%
  \makeatother%
  \begin{picture}(1,0.59845971)%
    \lineheight{1}%
    \setlength\tabcolsep{0pt}%
    \put(0,0){\includegraphics[width=\unitlength,page=1]{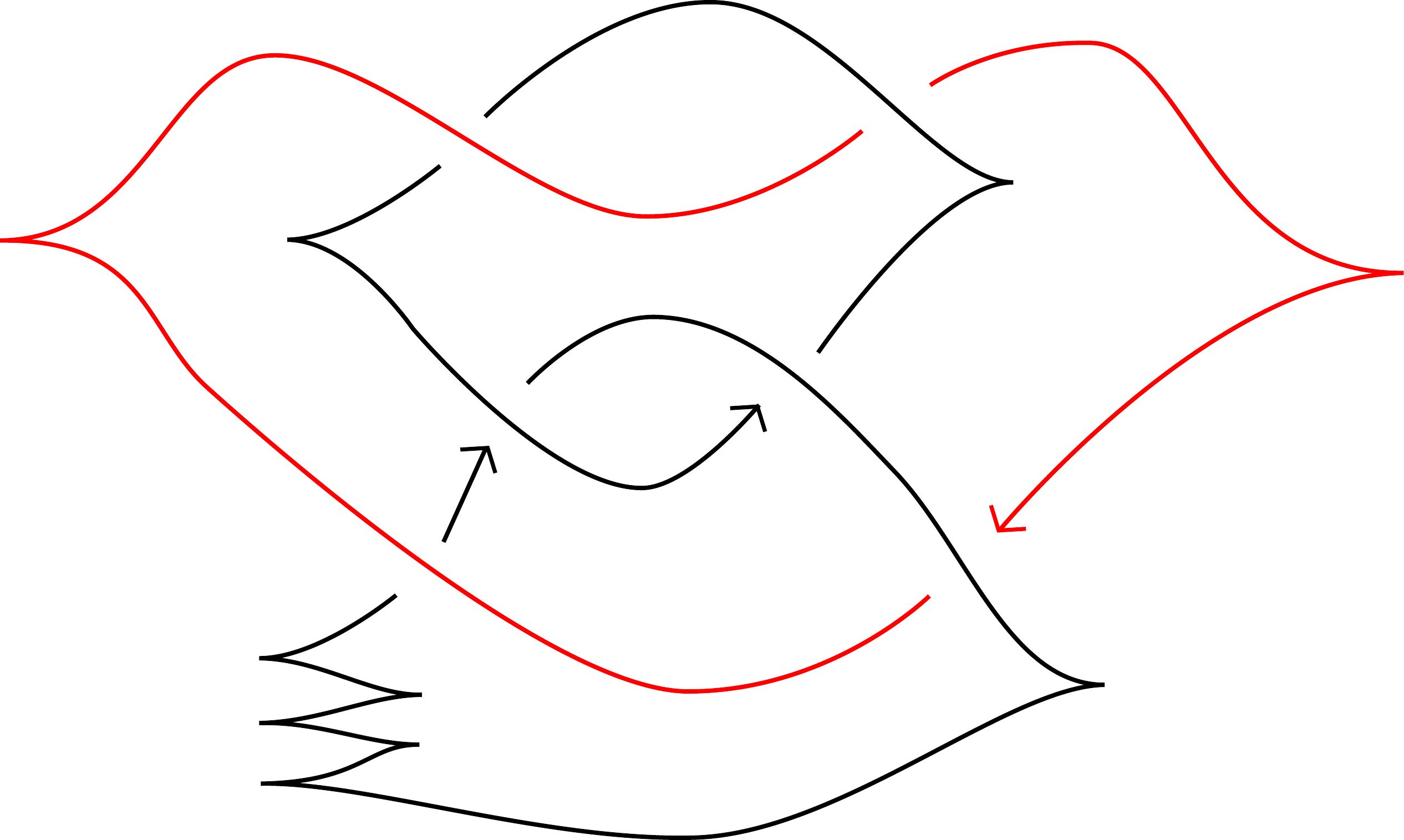}}%
    \put(0.06214429,0.29750443){\color[rgb]{0,0,0}\makebox(0,0)[lt]{\lineheight{1.25}\smash{\begin{tabular}[t]{l}$\mathcal K$\end{tabular}}}}%
    \put(0.67220175,0.40172648){\color[rgb]{0,0,0}\makebox(0,0)[lt]{\lineheight{1.25}\smash{\begin{tabular}[t]{l}$(-1)$\end{tabular}}}}%
    \put(0.77291526,0.14438547){\color[rgb]{0,0,0}\makebox(0,0)[lt]{\lineheight{1.25}\smash{\begin{tabular}[t]{l}$(-1)$\end{tabular}}}}%
  \end{picture}%
\endgroup%

     \caption{\smaller[1]{The regular fibre $K$ of $Y_0=M\left(-1;\frac{1}{3},\frac{1}{5}\right)\cong L(7,2)$ (left), and its Legendrian realisation $\mathcal K$ in $(Y_0,\zeta_{\text{can}})$ (right).}}
     \label{resolution}
 \end{figure}      

 We see the knot $\pm K$ as the transverse push-off of $\pm\mathcal K$ \cite{Baptiste}. Since \cite[Theorem 1.3]{AC} tells us that $K$ and $-K$ satisfy equality in the $\tau$-Bennequin bound, the above argument leads to \[\tb_{\zeta_{\text{can}}}(\mathcal K)-\rot_{\zeta_{\text{can}}}(\mathcal K)=\self_{\zeta_{\text{can}}}(K)=2\tau_{\zeta_{\text{can}}}(K)-1\] and 
 \[\begin{aligned}\tb_{\zeta_{\text{can}}}(\mathcal K)+\rot_{\zeta_{\text{can}}}(\mathcal K)&=\tb_{\zeta_{\text{can}}}(-\mathcal K)-\rot_{\zeta_{\text{can}}}(-\mathcal K)= \\ &=\self_{\zeta_{\text{can}}}(-K)=2\tau_{\zeta_{\text{can}}}(-K)-1=2\tau_{\overline\zeta_{\text{can}}}(K)-1\:,\end{aligned}\]
 which using Theorem \ref{teo:inequality} give \[\tb_{\zeta_{\text{can}}}(\mathcal K)=\text{TB}_{\zeta_{\text{can}}}(K)=\tau_{\zeta_{\text{can}}}(K)+\tau_{\overline\zeta_{\text{can}}}(K)-1\:.\] 
 Finally, we can then apply Equation \eqref{eq:tau} and write \[\begin{aligned}
     \text{tw}(Y_0&,\zeta_{\text{can}},K)=\text{TB}_{\zeta_{\text{can}}}(K)+e_1^TQ_0^{-1}e_1=-1+\tau_{\zeta_{\text{can}}}(K)+\tau_{\overline\zeta_{\text{can}}}(K)+e_1^TQ_0^{-1}e_1=\\ 
     &=-1+\left(-\dfrac{e_1^TQ_0^{-1}e_1}{2}+\dfrac{e_1^TQ_0^{-1}W_{\s_{\text{can}}}}{2}\right)+\left(-\dfrac{e_1^TQ_0^{-1}e_1}{2}+\dfrac{e_1^TQ_0^{-1}W_{\overline\s_{\text{can}}}}{2}\right)+e_1^TQ_0^{-1}e_1=\\ &=-1+e_1^TQ_0^{-1}\left(\dfrac{W_{\s_{\text{can}}}+W_{\overline\s_{\text{can}}}}{2}\right)=-1-\height[W_{\s_{\text{can}}}] \end{aligned}\] 
     from Equation \eqref{eq:height}, which concludes the first part of the statement.
 
 In order to prove the second part, we need to show that $\height[W_{\s_{\text{can}}}]\leq\height[W_{\s}]$ for every $\s\in\Spin^c(Y_0)$, and that the equality is realised when $\s$ is induced by a tight structure. When $e_0<-1$ there is nothing to prove, as in this case the standard graph describes the canonical Stein filling, and the coordinates of the characteristic vector for $\s$ induced by a tight contact structure are rotation numbers, which never equal the (negative of the) framing. Below we prove the statement for $e_0=-1$.

 The standard graph $\Gamma_0$ for $M(-1;r_1,r_2)$ is not minimal, having the central vertex $S_1$ with framing $m(1)=-1$; hence, there is a sub-graph, call it $\Gamma'_0$, that blows down when passing to a minimal resolution graph of a lens space. The sub-graph $\Gamma'_0$ then describes $S^3$, and therefore in order to have an initial characteristic vector whose full path ends correctly, its restriction to the vertices of $\Gamma'_0$ has to agree with $W_{\s_{\text{can}}}$. Furthermore, since the $\Spin^c$-structure on $S^3$ is self-conjugate, before reaching the terminal vector at least the coordinates corresponding to the vertices of $\Gamma'_0$ need to become as in the restriction of $-W_{\s_{\text{can}}}$. In the case of $W_\s$ for a $\Spin^c$-structure $\s$ induced by a tight contact structure, these are the only steps from the initial to the terminal vector (since coordinates on vertices neighbouring $\Gamma'_0$ are low enough that after blowing down, they correspond to rotation numbers); meanwhile, for any other $\s$ there are additional steps in the full path which may also lead to additional central steps (from $V$ to $V+2\text{PD}[S_1]$), causing a change of filtration level, and thus extending $\height[W_{\s}]$.

 More explicitly, we write the negative continued fraction expansions $-\frac{1}{r_1}=[m_1^1,\dots,m_{k_1}^1]$ and $-\frac{1}{r_2}=[m_1^2,\dots,m_{k_2}^2]$, and then their negative inverse is expressed by adding zero as the first term to the fractions, so that 
 $r_1=[0,m_1^1,\dots,m_{k_1}^1]$ and $r_2=[0,m_1^2,\dots,m_{k_2}^2]$. Denote \begin{equation}(\ell_1,\ell_2):=\max_{\substack{i_1\geq0\\ i_2\geq0}}\left\{(i_1,i_2)\:|\:\begin{aligned} &\ [0,m_1^1,\dots,m_{i_1}^1+1]+[0,m_1^2,\dots,m_{i_2}^2]=1 \\ &\ \hspace{3cm}\text{ or } \\ &\ [0,m_1^1,\dots,m_{i_1}^1]+[0,m_1^2,\dots,m_{i_2}^2+1]=1\end{aligned}\right\}\:.\label{eq:l}\end{equation}
 Then $\Gamma'_0$ consists of the vertices $\{S_1,S_1^1,\dots,S_{\ell_1}^1,S_1^2,\dots,S_{\ell_2}^2\}$, and the restriction of an initial vector to $\Gamma'_0$ is $(1,m_1^1+2,\dots,m_{\ell_1}^1 +2, m_1^2+2,\dots,m_{\ell_2}^2+2)$. We can assume that the full path starts with $S_1$ and continues with vertices of $\Gamma'_0$ until all of them satisfy the terminal condition; this requires as many central steps as 
 \begin{equation} 
  \#:=\label{eq:sharp}\text{denominator }[0,m_1^1,\dots,m_{\ell_1}^1] + \text{denominator }[0,m_1^2,\dots,m_{\ell_2}^2]-1,
 \end{equation} until we reach $(-1,-m_1^1-2,\dots,-m_{\ell_1}^1 -2, -m_1^2-2,\dots,-m_{\ell_2}^2-2)$. Afterwards, the coordinates at $S_{\ell_1+1}^1$ and $S_{\ell_2+1}^2$ are increased, and all other coordinates are as in the initial vector: the increase is by $2$ at $S_{\ell_1+1}^1$ and by $2\ell_1+2$ on $S_{1}^2$ if $\ell_2=0$; in  other cases it is by $2$ at $S_{\ell_1+1}^1$ and $4+2T$ at $S_{\ell_2+1}^2$ for $T$ the number of final $-2$'s in $[0,m_1^1,\dots,m_{\ell_1}^1]$ when we are in the first line of Equation \eqref{eq:l}, or with the role of the 1-st and 2-nd leg swapped for the second line. 
 
 In the case of $W_{\s_{\text{can}}}$, or any $W_\s$ for $\s$ induced by a tight structure, despite of the change in the coordinates we do not reach $-m_{\ell_1+1}^1$ or $-m_{\ell_2+1}^2$, nor the initial vector had any coordinate $v_j$ equal to $-m(j)$ for $j\neq1$, and the path ends. Otherwise, for general $\s$, additional steps are needed because some of the coordinates $v_j$, unchanged from the initial vector, satisfy $v_j=-m(j)$, or because they have become such by the correction above; these steps possibly lead to additional central steps, enlarging $\height[W_{\s}]$.
\end{proof}

\begin{remark}\label{rmk:tw}
 The fact that the canonical characteristic vector $V_{\emph{can}}$, such that $v_i=m(i)+2$ for each $i$, minimises the height of the full path, that is $\height[V_{\emph{can}}]\leq\height[V]$ for any $V$ whose full path ends correctly, holds generally for any Seifert fibred space $M(e_0;r_1,...,r_n)$ with negative-definite standard graph $\Gamma$. In the case of $e_0<-1$ the height is always $0$, and when $e_0=-1$, there is a sub-graph $\Gamma'_0$ representing $S^3$ and $m_1^i<-1-\#$ for $i>2$ (and also for $i=2$ when $\ell_2=0$) because of negative-definiteness; hence, we can conclude as in the proof of Theorem \ref{teo:tw}.
\end{remark}

\begin{proof}[Proof of Theorem \ref{cor:tw}]
For $e_0<-1$ we have shown in the proof of Proposition \ref{prop:tw} that the twisting number is equal to $-1$, and in the proof of Theorem \ref{teo:tw} that $\height[V_{\text{can}}]=0$.

For $M(-1;r_1,\dots,r_n)$ Proposition \ref{prop:tw} tells that the twisting number of the tight structures equals the one of $M(-1;r_1,r_2)$, which by the proof of Theorem \ref{teo:tw} equals $-1-\#$. Finally, that $\height[V_{\text{can}}]=\#$ follows from Remark \ref{rmk:tw}.
\end{proof} 

\section{Proof of the main results}
\label{section:five}
\subsection{Classification of tight structures}
\label{subsection:classification}
Once we have determined the twisting number, the classification follows the usual scheme: first, based on the work of Honda \cite{Honda1,Honda2}, or equivalently Giroux \cite{Giroux1, Giroux2}, an upper bound is found via convex decompositions, then it is shown that the upper bound is realised using explicit constructions.
\begin{prop}
 \label{prop:previous}  
 Suppose that $M=M(e_0;r_1,\dots,r_n)$ is a Seifert fibred space with negative-definite standard graph. The negative-twisting tight contact structures on $M$ are precisely the Legendrian surgeries on all possible Legendrian realisations of the complete blow-down of the standard graph.  
\end{prop}

\begin{proof}
 We can inductively expand the arguments of Ghiggini from \cite[Section 3]{G-} to an $n$-times punctured sphere $\Sigma_{0,n}$, and thus obtain that (up to an isotopy not necessarily fixed on the boundary) there is a unique tight contact structure on $\Sigma_{0,n}\times S^1$ with maximal twisting number $-q$ and  boundary slopes $-\frac{p_1}{q},\dots,-\frac{p_n}{q}$ with $p_n=-p_1-\cdots-p_{n-1}+n-2$. Following \cite{G-}, we will denote this unique tight structure by $\beta(p_1,\dots,p_{n-1};q)$. Furthermore, we call the \emph{background} of $(M,\xi)$ a piece of its convex decomposition which is equal to $(M,\xi)$ with neighbourhoods of singular fibres removed, so that vertical ruling curves on the boundary have twisting equal to $\mbox{tw}(M,\xi,K)$. The background is diffeomorphic to $\Sigma_{0,n}\times S^1$ and by the above uniqueness, for $\mbox{tw}(M,\xi,K)=-q$ it is contact isotopic to $\beta(p_1,\dots,p_{n-1};q)$. However, the trivialisation we use for its boundary tori is not the one from the product structure, but such that $-r_i$ is meridional slope of the glued-in torus; in this trivialisation, the boundary slopes equal $-\frac{p_1}{q},\dots,-\frac{p_n}{q}+e_0$. 

 When $e_0<-1$, the background is $\beta(1,1,\dots,1;1)$ and there are, by \cite[Theorem 2.3]{Honda1}, $\prod_{j=1}^{k_i}|m_j^i+1|$ possible tight contact structures on a solid torus glued in as a neighbourhood of the $i$-th singular fibre for $i=1,\dots,n-1$ and $|e_0+1|\prod_{j=1}^{k_n}|m_j^n+1|$ for the $n$-th fibre, where $-\frac{1}{r_i}=[m_1^i,\dots,m_{k_i}^i]$. Altogether, this gives the upper bound of \begin{equation}|e_0+1|\displaystyle\prod_{i=1}^n\displaystyle\prod_{j=1}^{k_i} |m_j^i+1|\label{eq:1}\end{equation} tight contact structures, which are realised by Legendrian surgery on all possible Legendrian realisations of the standard graph.

 When $e_0=-1$, we proved in Section \ref{section:four} that $q=1+\#$ as in Equation \eqref{eq:sharp} and the background is isotopic to $\beta(p,q-p,1,\dots,1;q)$. In fact, we can see more: \[-\frac{p}{q}=[m_1^1,\dots,m_{\ell_1}^1, \mu^1]^{-1} \text{ and }-\frac{q-p}{q}=[m_1^2,\dots,m_{\ell_2}^2, \mu^2]^{-1}\] where $(\mu^1,\mu^2)=(-2,-\ell_1-2)$ when $\ell_2=0$ (in this case $q=\ell_1+2,\ p=\ell_1+1$), and otherwise $(\mu^1,\mu^2)=(-2,-3-T)$ when $+1$ is in the first fraction in Equation \eqref{eq:l} and $(-3-T,-2)$ when $+1$ is in the second fraction; in both cases, $T$ is the number of final $-2$'s in the fraction $[m_1^\epsilon,\dots,m_{\ell_\epsilon}^\epsilon]$ where the $\epsilon$-th leg is the one with $+1$. Moreover, because of negative-definiteness,  if generators at the ($\ell_i+1$)-th position exist, they are $m_{\ell_1+1}^1<\mu^1$ or $m_{\ell_2+1}^2<\mu^2$, and $m_1^i$ for $i>2$ are smaller than $-q$. 

 The expression of boundary slopes of the background by truncated continued fractions makes it evident that not all basic slices in the decomposition of the $n$ solid tori are free. In fact, the solid tori that we glue onto the background in order to form $M$ will have boundary slope equal to $-\frac{p_i}{q}$, and meridional slope equal to $-r_i$ seen in the background basis, which on the tori (for $r_i=\frac{b_i}{a_i}$) corresponds to: \[A_{r_i}^{-1}\begin{pmatrix} p_i\\ q\end{pmatrix}:=\begin{pmatrix} a_i'&-b'_i\\ -a_i&  b_i\end{pmatrix}\begin{pmatrix} p_i\\ q\end{pmatrix}=\begin{pmatrix} \beta_i\\ \alpha_i\end{pmatrix}\ ,\] that is boundary slopes $-\frac{\alpha_1}{\beta_1}=[m_{k_1}^1,\dots,m_{\ell_1+2}^1,m_{\ell_1+1}^1-\mu^1], -\frac{\alpha_2}{\beta_2}=[m_{k_2}^2,\dots,m_{\ell_2+2}^2,m_{\ell_2+1}^2-\mu^2],$ and $-\frac{\alpha_i}{\beta_i}=[m_{k_i}^i,\dots,m_{1}^i+q]$ for $i>2$.  

 Exactly as many structures on $M$ can be obtained in $(S^3,\xist)$ by Legendrian realisations of the surgery diagram we get by blowing down the standard graph. In fact, set $\ell_i=0$ for $i=3,...,n$, then the number of structures when $e_0=-1$ is equal to \begin{equation} |m_{\ell_1+1}^1-\mu^1|\cdot|m_{\ell_2+1}^2-\mu^2|\cdot\prod_{i=3}^n|m_{1}^i+q|\cdot\prod_{i=1}^n\prod_{j=\ell_i+2}^{k_i}|m_j^i+1|\:.\label{eq:2}\end{equation} 
 Below we will denote $d_1:=\text{denominator }[0,m_1^1,\dots,m_{\ell_1}^1]$ and $d_2:=\text{denominator }[0,m_1^2,\dots,m_{\ell_2}^2]$. As we successively blow down the (truncated part of) the first two legs, this causes cabling of all the other singular fibres into parallel torus knots $T_{d_2,d_1}$. In the process, the first leg causes twisting in the meridional and the second leg in the longitudinal direction; if we imagine the torus link forming on the surface of a torus, the two unknots corresponding to the first two singular fibres appear as the belt circle above the torus surface, and as the core in the centre of the torus. We describe the result explicitly below.

 When $\ell_2=0$, we have surgery along an $(n-1,(n-1)d_1)$-cable of one component of standard $T_{2,2}$, with surgery coefficients $[m_1^i+q-1,m_2^i,\dots,m_{k_i}^i]$ for $i\geq 2$ on cabled components and $[m_{\ell_1+1}^1+1,m_{\ell_1+2}^1,\dots,m_{k_1}^1]$ on the remaining one.

 Meanwhile, in the general case the surgery diagram consists of a link whose $n$ components are standard torus knots: they are an $((n-2)d_2,\:(n-2)d_1)$-cable of the third component of a standard $T_{3,3}$ link, with surgery coefficients $[m_1^i+q-1,m_2^i,\dots,m_{k_i}^i]$ for $i\geq 3$, and such that each component of the cable has linking number $d_1$ with the second, and $d_2$ with the first, component of $T_{3,3}$. The surgery coefficients on the first and second component of $T_{3,3}$ are $[m_{\ell_1+1}^1-\mu^1-1,m_{\ell_1+2}^1,\dots,m_{k_1}^1]$ and $[m_{\ell_2+1}^2-\mu^2-1,m_{\ell_2+2}^2,\dots,m_{k_2}^2]$.
\end{proof}
We now complete the proof of Theorem \ref{teo:classification}; in particular, we show that if a contact structure $\xi$ on a Seifert fibred space $M$ as above has non-vanishing $c^+$ then it is negative-twisting. Note that this happens for example when $\xi$ is symplectically fillable. We recall that we introduced the concept of height of a full path class, with respect to $\mathcal F$, in Equation \eqref{eq:height}. 
\begin{proof}[Proof of Theorem \ref{teo:classification}]
 We have described all negative-twisting tight contact structures on considered manifolds already in Proposition \ref{prop:previous}. Since all of them are constructed by Legendrian surgery from $(S^3,\xist)$, they are all Stein fillable; in addition, since all structures on a given $M$ have smoothly the same filling and are induced by non-homotopic Stein structures on it, they are also distinguished  by the Heegaard Floer contact invariant $c^+$ (as proved by Plamenevskaya \cite{OlgaP}). 

 We now show that for all $\xi$ with $c^+(\xi)\neq 0$ the twisting is negative, which means $\text{tw}(M,\xi,K)<0$.
 Using Theorem \ref{teo:correction_term}, we get a canonical basis $\mathcal C=\{\gamma_1,...,\gamma_t\}$ of $\Imm\psi_*\subset\widehat{HF}(M,\s)$, determined by the full path algorithm (applied to the standard graph of $M$) as $\gamma_i=\psi_*[V_i]$ for $\{V_1,...,V_t\}$ initial vectors of the full paths that end correctly.  
 Since $\widehat c(\xi)\in\widehat{HF}(-M,\s)$ is also non-vanishing, we wish to compare $\tau_\xi(K)$ with $\tau_{\gamma_i}(K)$ for $\gamma_i\in\mathcal C$, for  $K$ the regular fibre of $M$; by \cite[Lemma 2.2]{AC} 
 \[\tau_\xi(K)=-\tau_{\widehat c(\xi)}(K)=\min\{\tau_\gamma(K)\:|\:\gamma\in\widehat{HF}(M,\s)\text{ and }\langle\gamma,\widehat{c}(\xi)\rangle=1\}\:;\] where we are using the canonical duality $\widehat{HF}(-M,\s)\simeq\widehat{HF}(M,\s)^\bullet$.
 Since $\rho^*:\widehat{HF}(-Y,\s)\rightarrow HF^+(-Y,\s)$ sends $\widehat c(\xi)$ into $c^+(\xi)\in\Ker U\subset HF^+(-Y,\s)$,  we  know from \cite[Proposition 3.2]{OSz-fullpath} that there exists at least one class $\gamma_a\in\mathcal C$ satisfying $\langle\gamma_a,\widehat c(\xi)\rangle=1$; analogously, we choose the class $\gamma_b=\J\gamma_a\in\mathcal C$ so that  $\langle\gamma_b,\widehat c(\overline\xi)\rangle=\langle\J\gamma_a,\J\widehat c(\xi)\rangle=1$. Therefore, we can write 
 \begin{equation}
 \tau_\xi(K)\leq\tau_{\gamma_a}(K)\hspace{1cm}\text{ and }\hspace{1cm}\tau_{\overline\xi}(K)\leq\tau_{\gamma_b}(K)\:. 
 \label{eq:gamma}
 \end{equation}
 From Equations \eqref{eq:tau} and \eqref{eq:gamma}, Lemma \ref{lemB}, Theorem \ref{teo:inequality} and Remark \ref{rmk:tw}, we obtain \[\begin{aligned}
  \dfrac{1}{-e(M)}&+\text{tw}(M,\xi,K)=\text{TB}_\xi(K)<\tau_\xi(K)+\tau_{\overline\xi}(K) \leq \\ &\leq\dfrac{1}{2}\left(\dfrac{1}{-e(M)}+e_1\cdot V_a\right)+ \dfrac{1}{2}\left(\dfrac{1}{-e(M)}+e_1\cdot V_b\right)=\dfrac{1}{-e(M)}-(\mathcal F(-V_b)-\mathcal F(V_a))= \\ &=\dfrac{1}{-e(M)}-\height[V_a]\leq\dfrac{1}{-e(M)}-\height[V_{\text{can}}]\:,
 \end{aligned}\] 
 where $\gamma_a$ and $\gamma_b$ are such that $\mathcal F(V_a)\leq\mathcal F(-V_b)$ which is possible up to swapping $\xi$ and $\overline\xi$. Since $\height[V_{\text{can}}]\geq0$ the claim follows. 
\end{proof}

 We can prove that $c^+(\xi)$ is actually a graded root \cite{Nemethi} of $HF^+$ for graph manifolds. This is done by generalising similar results of Karakurt \cite{Karakurt} and Bodn\'ar-Plamenevskaya \cite{BP}.
 
\begin{prop}
 \label{lemma:plus}   
 Let $Y$ be a three-manifold presented by an almost-rational graph $\Gamma$. Moreover, let $\xi=J\lvert_Y$ the restriction of a Stein structure obtained by interpreting the complete blow-down $X_\Gamma$ of $\Gamma$ as Stein $2$-handles attached to $D^4$. Then, there exists a characteristic vector $V$ whose full path ends correctly such that $c^+(\xi)=T_{[V]}\in HF^+(-Y,\s_\xi)$.
\end{prop}
\begin{proof}
  From the discussion in Subsection \ref{subsection:duality} above, we know that the family $\{T_{[V]}\}$ such that $V$ initiates a full path ending correctly is a basis of $\Ker U\subset HF^+(-Y)$, and the latter group contains $c^+(\xi)$. Note that the contact invariant is non-vanishing because $\xi$ is the boundary of the Stein domain $(X_\Gamma,J)$. We know that the first Chern class of the $\Spin^c$-structure induced by $J$ is determined by the rotation numbers in the $(-1)$-surgery presentation of $X_\Gamma$, which is given by blowing down the graph $\Gamma$. 
  
  Using a result of Plamenevskaya \cite{OlgaP}, we know that $J$ is the only $\Spin^c$-structure on $X_\Gamma$ such that $F^+_{\overline X_\Gamma,J}:HF^+(-Y,\s_\xi)\rightarrow HF^+(S^3)$ maps $c^+(\xi)$ to 1. In \cite[Section 5]{BP} there is the description of a function from $\text{Char}(\Gamma)$ to the set of the characteristic vectors on $X_\Gamma$, which is usually called \emph{strict transform} in literature and we denote it by $\mathcal{ST}$. It follows from the arguments in \cite{BP} that the preimage of the first Chern class of a Stein structure under $\mathcal{ST}$ is a full path; moreover, the strict transform induces (through the blow-downs) an automorphism of $HF^+(-Y,\s_\xi)$ which commutes with the cobordism maps, which means that if $V$ satisfies $c_1(J)=\mathcal{ST}(V)$ then $F^+_{\overline X_\Gamma,J}(T_{\mathcal{ST}[Z]})=1$ if and only if $[Z]=[V]$ for every $Z\in\text{Char}(\Gamma)$.
  
  The proof of our claim now follows as the one of \cite[Corollary 5.5]{BP}. Suppose that $c^+(\xi)=T_{[V_1]}+\cdots+T_{[V_k]}$, and let $\mathfrak u_i=c_1(V_i)$ and $\mathfrak v_i=c_1(\mathcal{ST}(V_i))$ for $i=1,...,k$; from what we said in the previous paragraph we have \[1=F^+_{\overline P_\Gamma,\mathfrak u_i}(c^+(\xi))=F^+_{\overline P_\Gamma,\mathfrak u_i}(T_{[V_i]})=F^+_{\overline X_\Gamma,\mathfrak v_i}(T_{\mathcal{ST}[V_i]})=F^+_{\overline X_\Gamma,\mathfrak v_i}(c^+(\xi))\] implying $\mathfrak v_1=\cdots=\mathfrak v_k=J$, and then $[V_1]=\cdots=[V_k]$; hence, we obtain that $c^+(\xi)$ cannot be written as a linear combination of more than one element of the basis $\{T_{[V]}\}$.
\end{proof}

Since (almost) every surgery on a torus link is a Seifert fibred space, Theorem \ref{teo:classification} immediately implies the following corollary. We assume $1\leq p\leq q$ coprime and $k\geq1$.

\begin{cor}
 \label{cor:link}
 Our classification holds for every surgery on a positive (resp. negative) torus link $T_{kp,\pm kq}$ whose surgery matrix $\Lambda$ is negative-definite (resp. has $\Lambda_{ii}>-pq$ and $\det(\Lambda)<0$).   
\end{cor} 
\begin{proof}
 It follows from Theorem \ref{teo:classification} once we prove that if the surgery matrix satisfies the assumption in the statement then the standard graph of the resulting 3-manifold is negative-definite.
 For simplicity, in this proof we assume that $r_i$ can be negative (for $T_{p,-q}$) when $i\geq3$. 
 A surgery on the torus link $T_{kp,\pm kq}$ with $q\geq p>0$ coprime and $k\geq1$ has graph obtained by adding $k$ legs to the graph representing the Seifert fibration of $T_{p,\pm q}$ in $S^3$, which is $M(-1;r_1,r_2)$ and the coefficients satisfy $r_1+r_2-1=\mp\frac{1}{pq}$ from Equation \eqref{eq:Brieskorn}. 
 
 Blowing down the graph, we obtain that the surgery matrix is \[\Lambda(r_3,...,r_{k+2})=D\pm pq\mathbf E\:,\] where $D=\text{Diag}(-\frac{1}{r_3},...,-\frac{1}{r_{k+3}})$ and $\mathbf E$ is the $k\times k$-matrix whose every entry is equal to one. Since in the case of $T_{p,q}$ we have that $\Lambda$ is negative-definite, its diagonal entries are negative; hence, we have $-\frac{1}{r_i}<-pq<0$ and then $0<r_i<\frac{1}{pq}$ (each negative framing on $T_{p,q}$ can be realised by some $(r_3,...,r_{k+2})$ in this way). In the case of $T_{p,-q}$ we just have $-\frac{1}{r_i}=\Lambda_{ii}+pq>0$ for every $i=3,...,k+2$.

 Since $D$ is then invertible, from standard linear algebra we can write \[\det(\Lambda)=\det(D)(1\pm pq\cdot\mathbf e^TD^{-1}\mathbf e)=\det(D)(1\mp pq(r_3+\cdots+r_{k+2}))\:,\] where $\mathbf e$ is the vector whose entries are all one. In addition, for $T_{p,q}$ the matrices $\Lambda$ and $D$ are both negative-definite, while for $T_{p,-q}$ the matrix $D$ is positive-definite. This means $\det(\Lambda)$ and $\det(D)$ have either the same sign (for $T_{p,q}$) or the opposite sign (for $T_{p,-q}$). We then conclude that $1\mp pq(r_3+\cdots+r_{k+2})$ is either positive or negative depending on the case, thus $r_3+\cdots+r_{k+2}<\pm\frac{1}{pq}$.

 The number $e(M)$ where $M:=S^3_\Lambda(T_{p,\pm q})=M(-1;r_1,...,r_{k+2})$ can then be estimated as follows: \[e(M)=-1+r_1+\cdots+r_{k+2}=\mp\dfrac{1}{pq}+r_3+\cdots+r_{k+2}<\mp\dfrac{1}{pq}\pm\dfrac{1}{pq}=0\:,\] which implies its standard graph is negative-definite. We recall that a Seifert fibred space has negative-definite standard graph if and only if $e(M)<0$, see \cite{Saveliev}. 
\end{proof}

\subsection{Obstructions to contact-type embeddings}
From Theorem \ref{teo:classification} and Proposition \ref{lemma:plus} we know that any contact structure with non-vanishing invariant $c^+(\xi)$ on $M$ has a contact surgery presentation given by the blown-down standard graph $\Gamma$, and $c^+(\xi)=T_{[V]}$ for a unique initial vector $V$ whose full path ends correctly. As described in Section \ref{section:four}, we can determine which vertices of the graph survive the blow-down procedure. Using the notation from Section \ref{section:four}, we denote by $Y_0$ the lens space with standard graph $\Gamma_0\subset\Gamma$, obtained by removing all the legs of $\Gamma$ except the two corresponding to the biggest coefficients $r_1$ and $r_2$, and by $\Gamma'\subset\Gamma_0$ the maximal sub-graph representing $S^3$. 

Passing to the minimal resolution, the sub-graph  $\Gamma'$ is blown-down completely, while $\Gamma_0$ becomes a minimal negative-definite graph for $Y_0$, and finally $\Gamma\setminus\Gamma_0$ is left unchanged except for the vertices which are connected to the centre of $\Gamma$, that get cabled into a torus link. 

The resulting surgery presentation gives a Stein filling $(X_\Gamma,J)$ for $(M,\xi)$ which consists only of Stein 2-handles, and it is negative-definite.
\begin{lemma}
 \label{lemma:main}
 Suppose that $M$ is a non-trivial Seifert fibred space whose standard graph is negative-definite. There is a fillable contact structure $\xi$ on $M$ such that $\xi$ is self-conjugate if and only if $Q_{X_\Gamma}$ is even, and in this case $d(M,s_\xi)$ is positive. 
\end{lemma}
\begin{proof}
 If we assume that there exists a self-conjugate $\xi$, all the framings involved in the minimal resolution are even, and each Legendrian knot appearing in the surgery presentation of $\xi$ is a standard torus knot with the same number of positive and negative stabilisations; in particular, the associated first Chern class $c_1(J)$ is the zero characteristic vector.
 The converse is trivial.

 It is straightforward to compute the $d_3$-invariant of the self-conjugate $\xi$. Since the 4-manifold $X_\Gamma$ is negative-definite, by Proposition \ref{prop:correction_term} we have \[d(M,\s_\xi)\geq d_3(\xi)=\dfrac{c_1^2(J)[X_\Gamma]+|\Gamma|-|\Gamma'|}{4}=\dfrac{|\Gamma|-|\Gamma'|}{4}>0\] as $c_1^2(J)[X_\Gamma]=0$. In conclusion, the correction term of a self-conjugate and fillable contact structure on $M$ is positive.
\end{proof}

We prove that the contact invariant of a fillable structure on $M$, inducing a spin structure, is either not $\Theta^+_{\s_\xi}$, which means is not in $U^n\cdot HF^+(-M,\s_\xi)$ for some $n\geq1$, or the correction term $d(M,\s_\xi)$ is positive. 

\begin{proof}[Proof of Proposition \ref{prop:main}]
 Suppose that $\xi$ is fillable and $c^+(\xi)=\Theta^+_{\s_\xi}$. Since by \cite[Theorem 2.10]{G-fillability} one has $c^+(\overline\xi)=\mathcal J\Theta^+_{\s_\xi}=\Theta^+_{\s_\xi}=c^+(\xi)$ as $\s_\xi$ is spin, and since by our classification in Theorem \ref{teo:classification} all the fillable structures on $M$ have distinct invariant, we have that $\xi$ is self-conjugate. Hence, it follows by Lemma \ref{lemma:main} that the correction term $d(M,\s_\xi)$ is positive.
\end{proof}
We now prove the first part of Theorem \ref{prop:convex}. We recall that a generalised $L$-space is a 3-manifold with trivial $HF_{\text{red}}$. 

\begin{lemma}
  \label{prop:convex1}
  Suppose that $M$ is a negative-definite Seifert fibred space. If $(M,\xi)$ admits a separating contact-type embedding in a strong symplectic filling $(X,\omega)$ of a generalised $L$-space then $\dim_\F\widehat{HF}_d(M,\s_\xi)=1$.
\end{lemma}
\begin{proof}
 Denote by $(M,\xi)$ the negative-definite Seifert fibred space, and suppose that the contact-type embedding exists. The strong symplectic filling $(X,\omega)$ is then split along $(M,\xi)$ in a convex part and a concave part: we call the first one $(F,\omega\lvert_F)$ which is a symplectic filling of $(M,\xi)$, and the second one $(C,\omega\lvert_C)$ which is a strong symplectic cobordism from $(M,\xi)$ to $(Y_0=\partial X,\s_0=J\lvert_{Y_0})$, where $J$ is an almost-complex structure compatible with $\omega$. Since by assumption $Y_0$ is a generalised $L$-space, and thus it has vanishing $HF_{\text{red}}(Y_0)$, we have that $c^+(\xi_0)=\Theta^+_{\s_0}$; therefore, from \cite[Theorem 1]{Etcheverria} and \cite[Theorem 2.1]{MT} we obtain that $c^+(\xi)=\Theta^+_{\s_\xi}$. 

 From our Theorem \ref{teo:classification} and Proposition \ref{lemma:plus} we know that $c^+(\xi)=\Theta^+_{\s_\xi}$ if and only if the dimension of $\widehat{HF}_d(-Y,\s_\xi)$ is one, where $d=d(Y,\s_\xi)$ is the correction term. 
\end{proof}

We can finish the proof of Theorem \ref{prop:convex}. We show that the same conclusion of Lemma \ref{prop:convex1} holds when $(M,\xi)$ is the boundary of a rationally convex domain in a closed K\"ahler manifold.
\begin{proof}[Proof of Theorem \ref{prop:convex}]
 The first part is precisely Lemma \ref{prop:convex1}. For the second part we follow Mark and Tosun argument in \cite[Theorem 1.7(b)]{MT3} and then using the results in \cite{BGS}, up to a small isotopy of the Stein filling $(F,J\lvert_F)$ of $(M,\xi)$, we can find a symplectic surface $\Sigma\subset \widehat X\setminus F$ which can be assumed to satisfy $n:=\Sigma\cdot\Sigma>2g-2$ when $g\geq1$. The neighbourhood $D_\Sigma$ of $\Sigma$ is a symplectic cap of the manifold $Y_{g,-n}$, the circle bundle over a surface of genus $g$ with Euler number $-n$. 
 
 We denote by $X$ the convex part of $\widehat X\setminus Y_{g,-n}$ which is then a strong symplectic filling of $Y_{g,-n}$. Ozsv\'ath and Szab\'o proved that when $|n|>2g-2$ the manifold $Y_{g,-n}$ is a generalised $L$-space; therefore, we can fall back to the case in Lemma \ref{prop:convex1}. 
\end{proof}

We conclude the paper by proving the application of our main results to Brieskorn spheres.
\begin{proof}[Proof of Corollary \ref{gompf}]
Let $Y=\Sigma(a_1,...,a_n)$ be a Brieskorn sphere with a separating contact-type embedding in a symplectic filling of a lens space. Integral homology spheres bounding negative-definite 4-manifolds have non-negative correction term \cite[Corollary 9.8]{OSz-negative}, this implies $d(Y)\geq0$. Meanwhile, since every symplectic filling of a lens space is negative-definite \cite{OSz-genus}, and such a filling exists for every lens space, we also have $d(-Y)\geq0$; therefore, we obtain $d(Y)=0$. From Theorem \ref{prop:convex} the contact 3-manifold $(Y,\xi)$ should satisfy $c^+(\xi)=\Theta^+$, but then the correction term should be positive because of Lemma \ref{lemma:main}. 
\end{proof}

\end{document}